\documentclass{amsart}

\usepackage[utf8]{inputenc}
\usepackage[english]{babel}
\usepackage{amsmath, amsfonts, amssymb}
\usepackage{caption}
\usepackage{cases}
\usepackage{enumitem}
\usepackage{graphicx}
\usepackage{hyperref}

\calclayout

\DeclareCaptionLabelFormat{andtable}%
  {\tablename~\thetable\ \& #1~#2}

\let\le\leqslant
\let\ge\geqslant
\let\emptyset\varnothing
\let\epsilon\varepsilon
\let\phi\varphi

\renewcommand{\d}{\mathrm{d}}
\renewcommand{\Im}{\mathop{\mathrm{Im}}\nolimits}
\renewcommand{\Re}{\mathop{\mathrm{Re}}\nolimits}

\newcommand{\bigmid}{\,\middle\vert\,}
\newcommand{\const}{\mathrm{const}}

\DeclareMathOperator{\Area}{Area}
\DeclareMathOperator{\am}{am}
\DeclareMathOperator{\cn}{cn}
\DeclareMathOperator{\Ec}{Ec}

\DeclareMathOperator{\Ind}{Ind}
\DeclareMathOperator{\Nul}{Nul}
\DeclareMathOperator{\sgn}{sgn}
\DeclareMathOperator{\sn}{sn}
\DeclareMathOperator{\tr}{tr}

\theoremstyle{plain}

\newtheorem{claim}{Claim}[section]
\newtheorem{conjecture}{Conjecture}[section]

\newtheorem{proposition}{Proposition}[section]
\newtheorem{theorem}{Theorem}

\theoremstyle{definition}

\theoremstyle{remark}
\newtheorem{remark}{Remark}[section]

\title[Rotationally symmetric critical metrics on tori in a conformal class]
{Rotationally symmetric critical metrics for Laplace eigenvalues on tori in a
conformal class}
\author{Egor Morozov}
\address{Universit\'e de Montr\'eal, Facult\'e des arts et des sciences,
D\'epartement de math\'ematiques et de statistique\endgraf
Pavillon Andr\'e-Aisenstadt (AA-5190)\endgraf
2920, chemin de la Tour\endgraf
Montr\'eal (QC)\endgraf
H3T 1J4}
\email{egor.morozov@umontreal.ca}

\begin{document}

\begin{abstract}
We study the problem of maximizing the first Laplace-Beltrami eigenvalue
normalized by area in a conformal class on a torus. By a result of Nadirashvili,
El Soufi, and Ilias, critical metrics for the \(k\)-th normalized
Laplace-Beltrami eigenvalue functional \(\bar\lambda_k\) in a conformal class
correspond to harmonic maps to spheres. In this paper we construct certain
\(\mathbb S^1\)-equivariant harmonic maps \(\mathbb T^2\to\mathbb S^3\). For
each non-rhombic conformal class on a torus, one of these maps corresponds to a
rotationally symmetric critical metric for \(\bar\lambda_1\) in this conformal
class with the value of \(\bar\lambda_1\) being greater than that of the flat
metric. This refines a recent result by Karpukhin that answers a question by El
Soufi, Ilias, and Ros. Also, we are able to show that if a rotationally
invariant metric on a rectangular torus is maximal for \(\bar\lambda_1\) in its
conformal class, then it is \(\mathbb S^1\)-equivariant and coincides (up to a
scalar factor) with the above metric. Finally, we show that a family of minimal
tori in \(\mathbb S^3\) called Otsuki tori fits naturally into our family. This
gives an explicit parametrization of Otsuki tori in terms of elliptic integrals.
\end{abstract}

\maketitle

\section{Introduction}

\subsection{Critical metrics for the normalized Laplace-Beltrami eigenvalues on
surfaces}

Let \((M,g)\) be a closed Riemannian surface and
\[
0=\lambda_0(M,g)<\lambda_1(M,g)\le\lambda_2(M,g)\le\ldots\le\lambda_k(M,g)\le
\ldots\nearrow +\infty
\]
be the Laplace-Beltrami eigenvalues of \((M,g)\). For each \(k\ge 1\) the
\textbf{\(k\)-th normalized eigenvalue functional} is defined by
\[
\bar\lambda_k(M,g):=\lambda_k(M,g)\Area(M,g).
\]
Consider the quantities
\[
\Lambda_k(M)=\sup_{g'\in\mathcal R(M)}\bar\lambda_k(M,g'),\qquad
\Lambda_k(M,[g])=\sup_{g'\in [g]}\bar\lambda_k(M,g'),
\]
where \(\mathcal R(M)\) is the set of all smooth Riemannian metrics on \(M\) and
\([g]=\{e^{2\omega}g\mid\omega\in C^\infty(M)\}\) is the conformal class of
\(g\). A metric \(g\in\mathcal R(M)\) is called \textbf{maximal} (respectively,
\textbf{maximal in the conformal class}) for the functional \(\bar\lambda_k\) if
\(\bar\lambda_k(M,g)=\Lambda_k(M)\) (respectively,
\(\bar\lambda_k(M,g)=\Lambda_k(M,[g])\)). A weaker notion is that of a
\emph{critical} (aka \emph{extremal}) metric introduced by
Nadirashvili~\cite{Nad1996} and refined by El~Soufi and
Ilias~\cite{ElSI2000,ElSI2003,ElSI2008}. A metric \(g\in\mathcal R(M)\) is
called \textbf{critical} (respectively, \textbf{critical in the conformal
class}) for the functional \(\bar\lambda_k\) if for any analytic deformation
\(g_t\subset\mathcal R(M)\) (respectively, \(g_t\subset\mathcal [g]\)), \(t\in
(-\epsilon,\epsilon)\) such that \(g_0=g\) we have
\[
\left.\frac{\d}{\d t}\right|_{t=0+}\bar\lambda_k(M,g_t)\cdot
\left.\frac{\d}{\d t}\right|_{t=0-}\bar\lambda_k(M,g_t)\le 0.
\]
In this paper we are interested in the case when \(M=\mathbb T^2\) is a torus.
For each \(a,b\in\mathbb R,b>0\) let \(g_{a,b}\) denote the flat metric on
\(\mathbb T^2\) induced by the factorization of \(\mathbb R^2\) by the lattice
\(\Gamma_{a,b}:=\mathbb Z(1,0)+\mathbb Z(a,b)\). Recall that the moduli space of
conformal classes on \(\mathbb T^2\) is given by \(\{[g_{a,b}]\mid
(a,b)\in\mathcal M\}\), where
\[
\mathcal M:=\{(a,b)\in\mathbb R^2\mid 0\le a\le 1/2,b>0,a^2+b^2\ge 1\}.
\]
In~\cite{ElSIR1996} the authors show that
\[
\Lambda_1(\mathbb T^2,[g_{a,b}])=\frac{4\pi^2}b,\qquad (a,b)\in\mathcal M,\quad
a^2+b^2=1,
\]
where the supremum is attained only on the flat metric. They also ask if the
flat metric is maximal for the functional \(\bar\lambda_1\) in other conformal
classes. In~2014, Petrides~\cite[Theorem~1]{Pet2014} showed that
\begin{equation}\label{eq:intro-petrides}
\Lambda_1(\mathbb T^2,[g_{a,b}])>8\pi,\qquad (a,b)\in\mathcal M,
\end{equation}
which implies that the flat metric \(g_{a,b}\) cannot be maximal in its
conformal class as soon as \(b\ge\frac{\pi}2\). Finally, recently
Karpukhin~\cite[Theorem~1.1]{Kar2025} showed that the metric \(g_{a,b}\) is not
maximal in its conformal class for all \((a,b)\in\mathcal M\) with
\(a^2+b^2>1\). The proof is based on the second variation formula for
\(\bar\lambda_1\) applied to a certain perturbation of the flat metric. In
particular, the method does not give an explicit example of a critical metric in
the conformal class with a larger value of \(\bar\lambda_1\) than that of the
flat metric. Our first main result is an explicit construction of such a
critical metric.

\begin{theorem}\label{th:intro-flatnotmax}
For any \((a,b)\in\mathcal M\) such that \(a^2+b^2>1\) there exists a
rotationally symmetric critical metric \(g_{a,b}^{1,1,0}=\rho(y)g_{a,b}\) for
\(\bar\lambda_1\) in the conformal class \([g_{a,b}]\) such that
\[
\bar\lambda_1(\mathbb T^2,g_{a,b}^{1,1,0})>
\max\Bigl\{\frac{4\pi^2}b,8\pi\Bigr\}.
\]
\end{theorem}

Our second main result concerns maximal metrics in the rectangular conformal
classes.

\begin{theorem}\label{th:intro-rectmax}
Let \(b>1\) and \(g=\rho(x,y)g_{0,b}\), where \(\rho\colon\mathbb R^2\to\mathbb
R\) is a \(\Gamma_{0,b}\)-periodic positive smooth function that does not depend
either on \(x\) or on \(y\). If \(g\) is maximal for the functional
\(\bar\lambda_1\) in the conformal class \([g_{0,b}]\), then
\(g=cg_{0,b}^{1,1,0}\) for some constant \(c>0\).
\end{theorem}

\begin{remark}
As seen from Theorem~\ref{th:intro-rectmax}, by a ``rotationally symmetric
metric'' in the conformal class \([g_{a,b}],(a,b)\in\mathcal M\) we mean a
metric of the form \(\rho(bx-ay,y)g_{a,b}\), where \(\rho\) is a smooth positive
function independent of one of its arguments. Note that this is not equivalent
to being invariant under an \(\mathbb S^1\) action by conformal diffeomorphisms.
The metrics with the latter property form a larger class.
\end{remark}

It is reasonable to expect for a maximal metric to have a rotational symmetry.
Since by Theorem~\ref{th:intro-flatnotmax} the metric \(g_{0,b}^{1,1,0}\) is
better than the flat one and is consistent with Petrides
inequality~\eqref{eq:intro-petrides}, we have the following natural conjecture.

\begin{conjecture}\label{con:intro-maxmetric}
For each \((a,b)\in\mathcal M\) such that \(a^2+b^2>1\) the metric
\(g_{a,b}^{1,1,0}\) is maximal for the functional \(\bar\lambda_1\) in the
conformal class \([g_{a,b}]\).
\end{conjecture}

Cf.~\cite[Conjecture~1]{KM2022} for the case of first normalized Steklov
eigenvalue in the conformal class. See also~\cite{Kang2025} for the
state-of-the-art upper bound on \(\Lambda_1(M,[g_{a,b}])\).

In~\cite{KLO2017}, the authors provide a method to compute numerically the
values of \(\Lambda_1(\mathbb T^2,[g_{a,b}])\) together with the corresponding
conformal factor. According to the paper, the optimal conformal factor is
rotationally symmetric but with a different direction of rotation (namely, the
conformal factors there have the form \(\hat\rho(bx-ay)\) for some positive
function \(\hat\rho\)). However, Professor Osting and Professor Chiu-Yen Kao
confirmed (private communication) that the code used in the paper required a
minor correction, and the true direction of rotation is the same as for the
metric \(g_{a,b}^{1,1,0}\). Also, for few test conformal classes considered
in~\cite{KLO2017}, the value of \(\Lambda_1(\mathbb T^2,[g_{a,b}])\) computed
numerically coincides (up to a numerical error) with \(\bar\lambda_1(\mathbb
T^2,g_{a,b}^{1,1,0})\) computed numerically using~\eqref{eq:intro-barlval}
below. Given all this, we may conclude that Conjecture~\ref{con:intro-maxmetric}
has a solid numerical evidence.

Keeping Conjecture~\ref{con:intro-maxmetric} in mind, it is interesting to study
the dependence of \(\bar\lambda_1(\mathbb T^2,g_{a,b}^{1,1,0})\) on
\((a,b)\in\mathcal M\). We are able to prove

\begin{theorem}\label{th:intro-ab}
For \((a,b)\in\mathcal M\), \(\bar\lambda_1(\mathbb T^2,g_{a,b}^{1,1,0})\) is an
increasing function in \(a\) and a decreasing function in \(b\).
\end{theorem}

Recall from~\cite{Nad1996} that the flat metric on the equilateral torus is a
unique global maximizer for \(\bar\lambda_1\). This fact is natural in light of
Conjecture~\ref{con:intro-maxmetric} and Theorem~\ref{th:intro-ab}, since the
conformal class of the equilateral torus is precisely the one with maximal \(a\)
and minimal \(b\).

\smallskip

The metric \(g_{a,b}^{1,1,0}\) is just one metric from a large family of metrics
\(g_{a,b}^{p,q,r}\), where \(p,q,r\) are integers satisfying certain
inequalities. To construct these metrics we utilize a well-known connection
between critical metrics in a conformal class and harmonic maps to spheres.
Recall that for a Riemannian surface \((M,g)\) the \textbf{Weyl's counting
function} is defined by
\[
N(\lambda):=\#\{i\mid\lambda_i(M,g)<\lambda\}.
\]

\begin{theorem}[{\cite[Theorems~3.1(ii) and~4.1(ii)]{ElSI2008}}]
\label{th:intro-NadElSI}
Let \(n\ge 1\), \(M\) be a smooth surface, and \(\mathbb S^n\) be the unit
sphere with the standard round metric \(g_0\).

\begin{enumerate}[label=(\arabic*)]
\item\label{it:intro-globcritnc}
Let \(f\colon M\to\mathbb S^n\) be a minimal map. Then the metric \(f^*g_0\)
is critical for the functional \(\bar\lambda_{N(2)}\) on \(M\).

\item\label{it:intro-confcritnc}
Let \(f\colon (M,g)\to\mathbb S^n\) be a harmonic map. Then the metric
\(\frac{1}2|\d f|_g^2 g\) is critical in the conformal class \([g]\) for the
functional \(\bar\lambda_{N(2)}\) on \(M\).
\end{enumerate}
\end{theorem}

In the setting of part~\ref{it:intro-confcritnc}, the number \(N(2)\) for the
metric \(\frac{1}2|\d f|_g^2 g\) is also called the \textbf{spectral index} of
\(f\) and denoted by \(\mathrm{ind}_S f\) (see~\cite[Definition~2.7]{Kar2021}).
See also~\cite{Pen2013a,Pen2019} for more details on the connection between
critical metrics and minimal and harmonic maps and related topics.

Part~\ref{it:intro-globcritnc} of Theorem~\ref{th:intro-NadElSI} was
successfully used many times to construct critical (and even maximal) metrics
explicitly on torus \cite{Nad1996,Lap2008,Pen2012,Pen2013b,Pen2015,Kar2014},
Klein bottle \cite{Lap2008,Pen2012,JNP2006,ElSGJ2006,CKM2019}, and orientable
surface of genus two \cite{JLNNP2005,NS2019}. In this paper we use
part~\ref{it:intro-confcritnc} of Theorem~\ref{th:intro-NadElSI} to construct
certain critical metrics \emph{in the conformal class} on torus. More precisely,
each metric \(g_{a,b}^{p,q,r}\) comes from a certain explicitly constructed
\(\mathbb S^1\)-invariant harmonic map \(u_{a,b}^{p,q,r}\) from \((\mathbb
T^2,[g_{a,b}])\) to \(\mathbb S^3=\{(z_1,z_2)\in\mathbb C^2\colon
|z_1|^2+|z_2|^2=1\}\), where the \(\mathbb S^1\) actions on \(\mathbb S^3\) and
\((\mathbb T^2,[g_{a,b}])\) are given by
\begin{align}
\alpha\cdot (z_1,z_2)&=(z_1,e^{2\pi i\alpha}z_2),&&(z_1,z_2)\in\mathbb
S^3,\label{eq:intro-actS3}\\
\alpha\cdot (x,y)&=(x+\alpha,y),&&
(x,y)\in\mathbb R^2/\Gamma_{a,b}\cong (\mathbb T^2,[g_{a,b}]),\qquad
\alpha\in\mathbb R/\mathbb Z\cong\mathbb S^1.
\label{eq:intro-actT2}
\end{align}
Note that the action~\eqref{eq:intro-actS3} is the same that is used in the
Hsiang-Lawson construction of \emph{Otsuki tori}, a family of \(\mathbb S^1\)
invariant \emph{minimal} tori in \(\mathbb S^3\) first considered by
Otsuki~\cite{Ots1970}. The initial construction was simplified by Hsiang and
Lawson~\cite{HL1971}. Later, Penskoi~\cite{Pen2013b} and
Karpukhin~\cite{Kar2014} used this family to construct critical metrics for some
of the functionals \(\bar\lambda_k\) on tori. We show in sec.~\ref{sec:otsuki}
that Otsuki tori fit naturally into our family. In particular, this gives an
explicit parametrization for Otsuki tori by functions that are expressed in
terms of elliptic integrals.

Unfortunately, we are not able to compute the spectral index for all the maps
\(u_{a,b}^{p,q,r}\) (in fact, we doubt if an answer in a closed form is possible
in the general case). However, we are able to give a lower bound together with
sufficient conditions for this bound to attain. In particular, we are able to
show that \(\mathrm{inf}_S u_{a,b}^{1,1,0}=1\), which is enough for
Theorem~\ref{th:intro-flatnotmax}.

\subsection{The main technical results}

Proceed to the definition of the maps \(u_{a,b}^{p,q,r}\). We need the following
auxiliary proposition.

\begin{proposition}\label{pr:intro-taus}
Let \(a,b\in\mathbb R,b>0\). Let \(p,q,r\) be a triple of integers such that
\(p,q\) are positive and
\begin{equation}\label{eq:intro-pqrineq}
\frac{p}q\ge\frac{1}2,\qquad
\left|\frac{r+a}q\right|\le\frac{1}2,\qquad
(r+a)^2+b^2>p^2.
\end{equation}
Then there exists a unique triple of real numbers \(\tau_1,\tau_2,\tau_3\)
satisfying
\begin{equation}\label{eq:intro-tausineq}
0\le\tau_1<\tau_2\le 1\le\tau_3,\qquad\tau_2<\tau_3
\end{equation}
and
\begin{gather}
\int_{\tau_1}^{\tau_2}\frac{\d t}{\sqrt{(t-\tau_1)(\tau_2-t)(\tau_3-t)}}=
\frac{2\pi b}q,\label{eq:intro-taub}\\
\int_{\tau_1}^{\tau_2}
\frac{\sqrt{\tau_1\tau_2\tau_3}\,\d t}{t\sqrt{(t-\tau_1)(\tau_2-t)(\tau_3-t)}}=
\frac{2\pi p}q,\label{eq:intro-taup}\\
\int_{\tau_1}^{\tau_2}
\frac{\sqrt{(1-\tau_1)(1-\tau_2)(\tau_3-1)}\,\d t}
{(1-t)\sqrt{(t-\tau_1)(\tau_2-t)(\tau_3-t)}}=\left|\frac{2\pi (r+a)}q\right|.
\label{eq:intro-taur}
\end{gather}
Moreover, \(\tau_1=0\) if and only if \(\frac{p}q=\frac{1}2\), in which case the
integral in~\eqref{eq:intro-taup} is treated as its own limit as \(\tau_1\to
0\). Similarly, \(\tau_2=1\) if and only if
\(\bigl|\frac{r+a}q\bigr|=\frac{1}2\), in which case the integral
in~\eqref{eq:intro-taur} is treated as its own limit as \(\tau_2\to 1\).

The converse is also true: if \(\tau_1,\tau_2,\tau_3\)
satisfy~\eqref{eq:intro-tausineq}, and \(p,q>0,r\) are integers
satisfying~\eqref{eq:intro-taub}--\eqref{eq:intro-taur}, then \(p,q,r\) must
satisfy~\eqref{eq:intro-pqrineq}. 
\end{proposition}

Let \(a,b,p,q,r,\tau_1,\tau_2,\tau_3\) be as in Proposition~\ref{pr:intro-taus}.
Put
\begin{equation}\label{eq:intro-mn}
m:=\frac{\tau_2-\tau_1}{\tau_3-\tau_1},\quad
n_0:=-\frac{\tau_2-\tau_1}{\tau_1},\quad
n_1:=\frac{\tau_2-\tau_1}{1-\tau_1}.
\end{equation}
Consider the map
\[
\tilde u_{a,b}^{p,q,r}\colon\mathbb R^2\to\mathbb S^3,\qquad
\tilde u_{a,b}^{p,q,r}(x,y)=(\cos\phi(y)e^{i\theta(y)},\sin\phi(y)e^{i(2\pi
x+\alpha(y))}),
\]
where \(\phi(y)\) is given by
\begin{subnumcases}{\phi(y)=}
\arccos\sqrt{(\tau_2-\tau_1)\sn^2(2\pi\sqrt{\tau_3-\tau_1}y\mid m)+
\tau_1},
&\(\frac{p}q>\frac{1}2,\quad\bigl|\frac{r+a}q\bigr|<\frac{1}2,\)\label{eq:intro-phigen}\\
\arccos(\sqrt{\tau_2}\sn(2\pi\sqrt{\tau_3}y\mid m)),
&\(\frac{p}q=\frac{1}2,\quad\bigl|\frac{r+a}q\bigr|<\frac{1}2,\)\label{eq:intro-philim1}\\
\arcsin(\sqrt{1-\tau_1}\cn(2\pi\sqrt{\tau_3-\tau_1}y\mid m)),
&\(\frac{p}q>\frac{1}2,\quad\bigl|\frac{r+a}q\bigr|=\frac{1}2,\)\label{eq:intro-philim2}\\
\frac{\pi}2-\am(2\pi\sqrt{\tau_3}y\mid m),
&\(\frac{p}q=\bigl|\frac{r+a}q\bigr|=\frac{1}2\)\label{eq:intro-philim12},
\end{subnumcases}
\(\theta(y)\) is given by
\begin{numcases}{\theta(y)=}
\sqrt{\frac{\tau_2\tau_3}{\tau_1(\tau_3-\tau_1)}}
\Pi(n_0;\am(2\pi\sqrt{\tau_3-\tau_1}y\mid m)\mid m),
&\(\frac{p}q>\frac{1}2,\)\label{eq:intro-theta}\\
\const,&\(\frac{p}q=\frac{1}2,\)\notag
\end{numcases}
and \(\alpha(y)\) is given by
\begin{numcases}{\alpha(y)=}
-\sgn(r+a)\sqrt{\frac{(1-\tau_2)(\tau_3-1)}{(1-\tau_1)(\tau_3-\tau_1)}}
\Pi(n_1;\am(2\pi\sqrt{\tau_3-\tau_1}y\mid m)\mid m),
&\(\bigl|\frac{r+a}q\bigr|<\frac{1}2,\)\label{eq:intro-alpha}\\
\const,
&\(\bigl|\frac{r+a}q\bigr|=\frac{1}2.\)\notag
\end{numcases}
We are ready to state the key technical result of the paper.

\begin{theorem}\label{th:intro-main}
The map \(\tilde u_{a,b}^{p,q,r}\) descends to a harmonic map
\(u_{a,b}^{p,q,r}\colon (\mathbb T^2,[g_{a,b}])\to\mathbb S^3\) with energy
density given by
\[
e(u_{a,b}^{p,q,r})=\rho(y):=2\pi^2(-\cos 2\phi+\tau_1+\tau_2+\tau_3-1).
\]
The number \(N(2)\) for the metric \(g_{a,b}^{p,q,r}:=\rho(y)g_{a,b}\) satisfies
\begin{equation}\label{eq:intro-N2ineq}
N(2)\ge 2p-1+\delta_{2p,q}+2(\lceil
2|r+a|-1\rceil+\delta_{r+a,0}),
\end{equation}
and the value of the corresponding functional is
\begin{equation}\label{eq:intro-barlval}
\bar\lambda_{N(2)}(\mathbb T^2,g_{a,b}^{p,q,r})=4\pi^2 b(\tau_1+\tau_2-\tau_3)+
8\pi q\sqrt{\tau_3-\tau_1}E(m).
\end{equation}
The inequality~\eqref{eq:intro-N2ineq} turns into equality whenever
\begin{equation}\label{eq:intro-N2eqsc}
\tau_2+\tau_3-\tau_1\le 4.
\end{equation}
This condition is satisfied whenever
\begin{equation}\label{eq:intro-N2eqscsc}
\text{either}\quad
\frac{p}q>\frac{1}{\sqrt 3}\quad\text{or}\quad
\left|\frac{r+a}q\right|<\frac{\sqrt 3}4.
\end{equation}
Moreover, for \(p=q=1,r=0\) we have
\begin{equation}\label{eq:intro-barlineq}
\bar\lambda_1(\mathbb T^2,g_{a,b}^{1,1,0})>\max\Bigl\{\frac{4\pi^2}b,
8\pi\Bigr\}\ge
\bar\lambda_1(\mathbb T^2,g_{a,b}).
\end{equation}
\end{theorem}

The cases \(\frac{p}q=\frac{1}2\) and \(\bigl|\frac{r+a}q\bigr|=\frac{1}2\) are
called the \textbf{first} and the \textbf{second limit cases} respectively. The
\textbf{third limit case}, i.e. what happens when the third inequality
in~\eqref{eq:intro-pqrineq} becomes equality, is treated in sec.~\ref{sec:lim3}.
Theorem~\ref{th:intro-flatnotmax} follows directly from
Theorem~\ref{th:intro-main} with \(p=q=1,r=0\). Note that for \(a=\frac{1}2\)
this corresponds to the second limit case.

Here are some important points about Proposition~\ref{pr:intro-taus} and
Theorem~\ref{th:intro-main}, which may help to gain some intuition about what is
going on.

\begin{itemize}
\item
Although we do not assume that \((a,b)\in\mathcal M\), it is always so in the
case \(p=q=1,r=0\) because of~\eqref{eq:intro-pqrineq}.

\item
In the first limit case the map \(u_{a,b}^{p,q,r}\) is not full, i.e. the image
is contained in an equatorial \(\mathbb S^2\subset\mathbb S^3\).

\item
In the second limit case we have either \(a\in\mathbb Z\) and \(q\) is even or
\(a\notin\mathbb Z,2a\in\mathbb Z\) and \(q\) is odd.

\item
The ``hybrid'' limit case \(\frac{p}q=\bigl|\frac{r+a}q\bigr|=\frac{1}2\) is
possible only for \(a\in\mathbb Z\). The maps \(u_{0,b}^{1,2,1}\) coincide with
the Gauss maps of Delaunay nodoids and unduloids,
see~\cite[sec.~4]{Eel1987},\cite[Theorem~4.3.2]{Smith1972}. When \(b\to+\infty\)
(equivalently, \(\tau_3,m\to 1\)), the map \(u_{0,b}^{1,2,1}\) converges to the
Gauss map of the catenoid.

\item
The inequality~\eqref{eq:intro-N2ineq} can be strict, see
Remark~\ref{rem:weyl-N2noteq}.
\end{itemize}

Let \(a,b\in\mathbb R,b>0\) and let \(p>0,r\) be integers satisfying
\begin{equation}\label{eq:intro-preq}
(r+a)^2+b^2=p^2.
\end{equation}
Then for each \(\phi_0\in\bigl[0,\frac{\pi}2\bigr]\) the map
\[
\tilde u_{a,b}^{p,r;\phi_0}\colon\mathbb R^2\to\mathbb S^3,\qquad
\tilde u_{a,b}^{p,r;\phi_0}(x,y):=
(\cos\phi_0\,e^{\frac{2\pi i}b py},\sin\phi_0\,e^{\frac{2\pi i}b (bx-(r+a)y)})
\]
descends to a harmonic map \(u_{a,b}^{p,r;\phi_0}\colon (\mathbb
T^2,[g_{a,b}])\to\mathbb S^3\). Our classification result is the following.

\begin{theorem}\label{th:intro-harm}
Let \(a,b\in\mathbb R,b>0\) and \(u\colon (\mathbb T^2,[g_{a,b}])\to\mathbb
S^3\) be a harmonic map that is \(\mathbb S^1\)-equivariant w.r.t. the
actions~(\ref{eq:intro-actS3},\ref{eq:intro-actT2}). Then there exist an
\(\mathbb S^1\)-equivariant conformal automorphism \(\gamma_1\colon (\mathbb
T^2,[g_{a,b}])\to (\mathbb T^2,[g_{a,b}])\) and an \(\mathbb S^1\)-equivariant
isometry \(\gamma_2\colon\mathbb S^3\to\mathbb S^3\) such that either
\(u=\gamma_2\circ u_{a,b}^{p,q,r}\circ\gamma_1\) for some integer \(p,q>0\) and
\(r\) satisfying~\eqref{eq:intro-pqrineq} or \(u=\gamma_2\circ
u_{a,b}^{p,r;\phi_0}\circ\gamma_1\) for some integer \(p>0,r\)
satisfying~\eqref{eq:intro-preq} and \(\phi_0\in\bigl[0,\frac{\pi}2\bigr]\).
\end{theorem}

\subsection{Organization of the paper}

The paper is organized as follows. Proposition~\ref{pr:intro-taus} and
Theorem~\ref{th:intro-harm} are proved in sec.~\ref{sec:taus} and~\ref{sec:harm}
respectively. The harmonicity of the maps \(u_{a,b}^{p,q,r}\) is a byproduct of
the proof. In sec.~\ref{sec:weyl} we prove the bound~\eqref{eq:intro-N2ineq}
together with sufficient
conditions~\eqref{eq:intro-N2eqsc},\eqref{eq:intro-N2eqscsc}. In
sec.~\ref{sec:barlval} we prove the inequality~\eqref{eq:intro-barlval} (which
finishes the proof of Theorem~\ref{th:intro-main}) and
Theorem~\ref{th:intro-rectmax}. In sec.~\ref{sec:otsuki} we show how Otsuki tori
fit into our picture and in sec.~\ref{sec:ab} we prove
Theorem~\ref{th:intro-ab}. The third limit case is studied in
sec.~\ref{sec:lim3}. Finally, in sec.~\ref{sec:disc} we discuss the difficulties
in proving Conjecture~\ref{con:intro-maxmetric} and sketch possible directions
for the future research.

\subsection*{Acknowledgements}

The author is grateful to Mikhail Karpukhin and Iosif Polterovich for reading
early versions of the manuscript and giving numerous suggestions on how to
improve it. Special thanks to Mikhail Karpukhin for
Proposition~\ref{pr:ab-hopf}. The author is also grateful to Braxton Osting,
Chiu-Yen Kao, and Alexei V.\,Penskoi for useful discussions. The work is
partially supported by the ISM Graduate Scholarship.

\section{Existence of \texorpdfstring{\(\tau_1,\tau_2,\tau_3\)}{τ₁, τ₂, τ₃}}
\label{sec:taus}

In this section we prove Proposition~\ref{pr:intro-taus}. For everything
concerning elliptic integrals the reader may consult~\cite{BF1971,Car2025,WR}.
Here and in sec.~\ref{sec:barlval} we relay on symbolic calculations made with
Wolfram Mathematica. Although it would be very tedious, in principle it is
possible to verify all the calculations manually.

By~\cite[233.00 and 233.02]{BF1971} we have
\begin{align}
\int_{\tau_1}^{\tau_2}\frac{\d t}{\sqrt{(t-\tau_1)(\tau_2-t)(\tau_3-t)}}&=
\frac{2}{\sqrt{\tau_3-\tau_1}}
K\left(\frac{\tau_2-\tau_1}{\tau_3-\tau_1}\right),\notag\\
\int_{\tau_1}^{\tau_2}
\frac{\sqrt{\tau_1\tau_2\tau_3}\,\d t}{t\sqrt{(t-\tau_1)(\tau_2-t)(\tau_3-t)}}&=
2\sqrt{\frac{\tau_2\tau_3}{\tau_1(\tau_3-\tau_1)}}
\Pi\left(-\frac{\tau_2-\tau_1}{\tau_1}\bigmid
\frac{\tau_2-\tau_1}{\tau_3-\tau_1}\right),\label{eq:taus-bf}\\
\int_{\tau_1}^{\tau_2}
\frac{\sqrt{(1-\tau_1)(1-\tau_2)(\tau_3-1)}\,\d t}
{(1-t)\sqrt{(t-\tau_1)(\tau_2-t)(\tau_3-t)}}&=
2\sqrt{\frac{(1-\tau_2)(\tau_3-1)}{(1-\tau_1)(\tau_3-\tau_1)}}
\Pi\left(\frac{\tau_2-\tau_1}{1-\tau_1}\bigmid
\frac{\tau_2-\tau_1}{\tau_3-\tau_1}\right).\notag
\end{align}
Hence, one can rewrite~\eqref{eq:intro-taub}--\eqref{eq:intro-taur} in terms
of~\eqref{eq:intro-mn} as
\begin{equation}\label{eq:taus-nm}
\begin{gathered}
\sqrt{\left(\frac{1}{n_1}-\frac{1}{n_0}\right)m}K(m)=\frac{\pi b}q,\quad
\sqrt{\frac{(1-n_0)(n_0-m)}{n_0}}\Pi(n_0\mid m)=\frac{\pi p}q,\\
\sqrt{\frac{(1-n_1)(n_1-m)}{n_1}}\Pi(n_1\mid m)=\left|\frac{\pi (r+a)}q\right|.
\end{gathered}
\end{equation}
Introduce the function
\begin{equation}\label{eq:taus-Phi}
\Phi(n\mid m):=
\sqrt{\frac{(1-n)(n-m)}n}\Pi(n\mid m).
\end{equation}
We have
\[
\frac{\partial\Phi}{\partial n}(n\mid m)=
\frac{nE(m)+(m-n)K(m)}{2n^2}\sqrt{\frac{n}{(1-n)(n-m)}}.
\]
Since
\begin{equation}\label{eq:taus-EKineq}
E(m)<K(m)<\frac{1}{1-m}E(m),\qquad m\in (0,1),
\end{equation}
it is easy to see that
\begin{equation}\label{eq:taus-EKpos}
nE(m)+(m-n)K(m)>0,\qquad m\in (0,1),\quad n\in (-\infty,0)\cup (m,1),
\end{equation}
and it follows that for \(m\in (0,1)\) fixed, \(\Phi(n\mid m)\) is an increasing
function in \(n\) on the intervals \((-\infty,0)\) and \((m,1)\). Also, we have
\begin{equation}\label{eq:taus-nlim}
\begin{aligned}
\Pi(n\mid m)&=\frac{\pi}2\frac{1}{\sqrt{-n}}+o\left(\frac{1}{\sqrt{-n}}\right),&
n&\to -\infty,\\
\Pi(n\mid m)&=\frac{\pi}2\sqrt{\frac{n}{(1-n)(n-m)}}+
o\left(\frac{1}{\sqrt{1-n}}\right),&
n&\to 1,
\end{aligned}
\end{equation}
and hence
\[
\lim_{n\to -\infty}\Phi(n\mid m)=\lim_{n\to 1}\Phi(n\mid m)=\frac{\pi}2.
\]
Since
\[
\lim_{n\to 0-}\Phi(n\mid m)=+\infty,\qquad
\Phi(m\mid m)=0,
\]
we conclude that for each \(m\in (0,1)\) there exists a unique pair of \(n_0\in
[-\infty,0)\) and \(n_1\in [m,1]\) such that
\begin{equation}\label{eq:taus-n0n1}
\Phi(n_0\mid m)=\frac{\pi p}q,\qquad
\Phi(n_1\mid m)=\left|\frac{\pi (r+a)}q\right|.
\end{equation}
In particular, we see that \(\frac{p}q=\frac{1}2\) if and only if
\(n_0=-\infty\) and \(\tau_1=0\). Similarly,
\(\bigl|\frac{r+a}q\bigr|=\frac{1}2\) if and only if \(n_1=1\) and \(\tau_2=1\).
Consider \(n_0,n_1\) as functions in \(m\) defined implicitly
by~\eqref{eq:taus-n0n1} and introduce the function
\begin{equation}\label{eq:taus-Psi}
\Psi(m):=\left(\frac{1}{n_1}-\frac{1}{n_0}\right)mK(m)^2.
\end{equation}
From the implicit function theorem we find
\begin{equation}\label{eq:taus-nprim}
n_i'=-\left.\frac{\partial\Phi}{\partial m}(n_i\mid m)\middle/
\frac{\partial\Phi}{\partial n}(n_i\mid m)\right.=
\frac{n_i(n_i-1)E(m)}{(m-1)(n_i E(m)+(m-n_i)K(m))},\qquad i=0,1,
\end{equation}
and
\[
\Psi'(m)=\frac{KE(K-E)((m-1)K+E)}{1-m}
\frac{n_1-n_0}{(n_0E+(m-n_0)K)(n_1E+(m-n_1)K)}>0,
\]
where we omit the argument \(m\) in \(K(m),E(m)\) and
use~\eqref{eq:taus-EKineq},\eqref{eq:taus-EKpos} in the latter inequality. Since
\[
\Pi(n\mid 0)=\frac{\pi}2\frac{1}{\sqrt{1-n}},
\]
it is easy to see that
\begin{equation}\label{eq:taus-limzero}
\lim_{m\to 0}\frac{m}{n_0}=1-\frac{4p^2}{q^2},\qquad
\lim_{m\to 0}\frac{m}{n_1}=1-\frac{4(r+a)^2}{q^2},\qquad
\lim_{m\to 0}\Psi(m)=\frac{\pi^2}{q^2}(p^2-(r+a)^2),
\end{equation}
and since
\[
\Pi(n\mid m)=-\frac{1}{2(1-n)}\ln(1-m)+o(\ln(1-m)),\qquad m\to 1,
\]
it is easy to see that
\begin{equation}\label{eq:taus-limone}
\lim_{m\to 1}n_0=-\infty,\qquad
\lim_{m\to 1}n_1=1,\qquad
\lim_{m\to 1}\Psi(m)=+\infty.
\end{equation}
Since \((r+a)^2+b^2>p^2\), there exists a unique \(m\in (0,1)\) such that
\(\Psi(m)=\frac{\pi^2 b^2}{q^2}\), which concludes the proof. The converse is
proved simply by reversing the argument above.\qed

\section{\texorpdfstring{\(\mathbb S^1\)}{S¹}-equivariant harmonic tori in
\texorpdfstring{\(\mathbb S^3\)}{S³}}\label{sec:harm}

In this section we classify all \(\mathbb S^1\)-equivariant harmonic tori in
\(\mathbb S^3\) w.r.t. the actions~(\ref{eq:intro-actS3},\ref{eq:intro-actT2}).
As a byproduct, we obtain that the maps \(u_{a,b}^{p,q,r}\) are harmonic.

The orbit space of the action~\eqref{eq:intro-actS3} identifies with
\[
\mathbb S_{\ge 0}^2=\{(x_1,x_2,x_3)\in\mathbb S^2\colon x_3\ge 0\}.
\]
The identification is given by
\[
\mathbb S_{\ge 0}^2\cong\mathbb S^3/\mathbb S^1,\qquad
(x_1,x_2,x_3)\leftrightarrow
\{(x_1+ix_2,e^{2\pi i\alpha}x_3)\in\mathbb S^3\mid
\alpha\in\mathbb R/\mathbb Z\}.
\]
The points on the boundary \(\partial\mathbb S_{\ge 0}^2\) (i.e. with \(x_3=0\))
correspond to zero dimensional orbits, i.e. fixed points of the \(\mathbb S^1\)
action. These are so-called \emph{exceptional orbits}. The points in \(\mathbb
S_{>0}^2:=\mathbb S_{\ge 0}^2\setminus\partial\mathbb S_{\ge 0}^2\) correspond
to one dimensional orbits. These are \emph{orbits of general type}.

\begin{proof}[Proof of Theorem~\ref{th:intro-harm}]
Since the orbit space of the action~\eqref{eq:intro-actT2} identifies with
\(\mathbb R/b\mathbb Z\), the map \(u\) descends to a map \(u_*\colon\mathbb
R/b\mathbb Z\to\mathbb S_{\ge 0}^2\) smooth on \(u_*^{-1}(\mathbb S_{>0}^2)\).
The image of this map is a curve \(\gamma_u\) in the orbit space. In the sequel,
we work with the lifted map \(\tilde u_*\colon\mathbb R\to\mathbb S_{\ge 0}^2\).
Introduce spherical coordinates on \(\mathbb S_{\ge 0}^2\),
\[
\mathbb S_{\ge 0}^2=\{(\cos\phi\cos\theta,\cos\phi\sin\theta,\sin\phi)\mid
\phi\in [0,\pi/2],\theta\in\mathbb R/2\pi\mathbb Z\}.
\]
Note that \(\theta\) is not unique at \(N:=(0,0,1)\in\mathbb S_{\ge 0}^2\).
Suppose for now that \(N\notin\gamma_u\) and \(\gamma_u\cap\partial\mathbb
S_{\ge 0}^2=\emptyset\) (we will see soon that these assumptions mean that we
are not dealing with the limit cases). Then \(\tilde u_*\) is given in spherical
coordinates by smooth functions \(\phi\colon\mathbb R\to (0,\pi/2)\) and
\(\theta\colon\mathbb R\to\mathbb R\). Since \(\gamma_u\) is closed, the
function \(\phi(y)\) is \(b\)-periodic and the function \(\theta(y)\) satisfies
\begin{equation}\label{eq:harm-thper}
\theta(y+b)\equiv\theta(y)+2\pi p
\end{equation}
for some \(p\in\mathbb Z\) (this property is sometimes called
\emph{quasiperiodicity}).

Here we come to a crucial difference between the minimal case of~\cite{HL1971}
and the harmonic case. An equivariant minimal surface is completely determined
by its image in the orbit space. However, it is not so in the harmonic case
since harmonicity is not invariant under a reparametrization of the domain.
Therefore, the full expression for the map \(u\) must include an additional
parameter controlling the coordinate in the orbit. This said, the lifted map
\(\tilde u\colon\mathbb R^2\to\mathbb S^3\) has the form
\[
\tilde u(x,y)=(\cos\phi(y)e^{i\theta(y)},\sin\phi(y)e^{i(2\pi
x+\alpha(y))})
\]
for some smooth function \(\alpha\colon\mathbb R\to\mathbb R\). Again, since
\(\gamma_u\) is closed, the function \(\alpha(y)\) satisfies
\begin{equation}\label{eq:harm-alper}
\alpha(y+b)\equiv\alpha(y)-2\pi (r+a)
\end{equation}
for some \(r\in\mathbb Z\). In the sequel, we often omit the argument \(y\) in
\(\phi(y),\theta(y),\alpha(y)\). We now use the following result

\begin{theorem}[{\cite[Theorem~1.3.2]{Smith1972}}]\label{th:harm-red}
Let \(G\) be a Lie group acting on Riemannian manifolds \(M,N\). A
\(G\)-equivariant map \(f\colon M\to N\) is harmonic if and only if its energy
is stationary w.r.t. all compactly supported equivariant variations.
\end{theorem}

The doubled energy density of \(u\) is given by
\begin{equation}\label{eq:harm-eu}
2e(u)=|\partial_x u|^2+|\partial_y u|^2=
(\phi')^2+(\theta')^2\cos^2\phi+((\alpha')^2+4\pi^2)\sin^2\phi,
\end{equation}
and the first equivariant variations of energy in \(\phi,\theta,\alpha\) are
given by
\begin{gather*}
\delta_\phi E(u)=-\phi''-(\theta')^2\sin\phi\cos\phi+
((\alpha')^2+4\pi^2)\sin\phi\cos\phi,\\
\delta_\theta E(u)=-(\theta'\cos^2\phi)',\qquad
\delta_\alpha E(u)=-(\alpha'\sin^2\phi)'.
\end{gather*}
By Theorem~\ref{th:harm-red}, we have
\begin{gather}
-\phi''-c^2\frac{\sin\phi}{\cos^3\phi}+d^2\frac{\cos\phi}{\sin^3\phi}+
4\pi^2\sin\phi\cos\phi=0,\label{eq:harm-phieq}\\
\theta'=\frac{c}{\cos^2\phi},\qquad
\alpha'=\frac{d}{\sin^2\phi}.\label{eq:harm-thaleq}
\end{gather}
for some constants \(c,d\in\mathbb R\). Multiplying~\eqref{eq:harm-phieq} by
\(\phi'\) and integrating, we obtain
\[
\phi'=\pm\sqrt{-\frac{c^2}{\cos^2\phi}-\frac{d^2}{\sin^2\phi}-
4\pi^2\cos^2\phi+A},
\]
where \(A\in\mathbb R\) is another constant. Finally, after the change of
variable \(\tau:=\cos^2\phi\) we arrive at
\begin{equation}\label{eq:harm-taueq}
\tau'=\mp 4\pi\sqrt{P(\tau)},
\end{equation}
where
\begin{equation}\label{eq:harm-P}
4\pi^2 P(t)=4\pi^2 t^3-(A+4\pi^2)t^2+(c^2-d^2+A)t-c^2.
\end{equation}

Let \(\tau_1,\tau_2\in (0,1)\) be the minimal and the maximal values of \(\tau\)
respectively. Note that the shifts \((x,y)\mapsto
(x+x_0,y+y_0),x_0,y_0\in\mathbb R\) give rise to \(\mathbb S^1\)-equivariant
conformal diffeomorphisms of \((\mathbb T^2,[g_{a,b}])\). Replacing \(u\) by its
precomposition with an appropriate conformal diffeomorphism of this kind, we may
assume that \(\phi(0)=\tau_1\) and \(\alpha(0)=0\). Similarly, the rotations
\((\phi,\theta)\mapsto (\phi,\theta+\theta_0),\theta_0\in\mathbb R\) and the
mirror symmetry \((\phi,\theta)\mapsto (\phi,-\theta)\) of the orbit space
\(\mathbb S_{\ge 0}^2\) give rise to \(\mathbb S^1\)-equivariant isometries of
\(\mathbb S^3\). Replacing \(u\) by its postcomposition with an appropriate
isometry of this kind, we may assume that \(\theta(0)=0\) and \(c\ge 0\) (in
particular, \(p\ge 0\)). After these modifications, our goal becomes just to
show that either \(u=u_{a,b}^{p,q,r}\), where \(q\in\mathbb N\) and \(p,q,r\)
satisfy~\eqref{eq:intro-pqrineq}, or \(u=u_{a,b}^{p,r;\phi_0}\), where
\(\phi_0\in\left[0,\frac{\pi}2\right]\) and \(p,r\)
satisfy~\eqref{eq:intro-preq}.

Suppose that \(\tau_1=\tau_2\), i.e. \(\tau\equiv\const\). Then
using~\eqref{eq:harm-thaleq} and \(\theta(0)=\alpha(0)=0\), we obtain
\(\theta(y)=\frac{c}{\tau_1}y\) and \(\alpha(y)=\frac{d}{1-\tau_1}y\). It
follows from~\eqref{eq:harm-thper} and~\eqref{eq:harm-alper} that
\(\frac{c}{\tau_1}=\frac{2\pi p}b\) and
\(\frac{d}{1-\tau_1}=-\frac{2\pi(r+a)}b\). We have \(u=u_{a,b}^{p,r;\phi_0}\),
where \(\phi_0=\arccos\sqrt{\tau_1}\), as desired.

Now suppose that \(\tau_1\ne\tau_2\). Then \(\tau_1,\tau_2\) are roots of the
cubic polynomial \(P(t)\). Since both roots are real, the third root \(\tau_3\)
is real as well. Since
\begin{equation}\label{eq:harm-cd}
\tau_1\tau_2\tau_3=-P(0)=\frac{c^2}{4\pi^2}\ge 0,\qquad
(1-\tau_1)(1-\tau_2)(1-\tau_3)=P(1)=-\frac{d^2}{4\pi^2}\le 0,
\end{equation}
it follows that \(\tau_3\ge 1\). It is clear that the function \(\phi(0)\)
increases from the value \(\tau_1\) at \(y=0\) until it achieves the value
\(\tau_2\) at some \(y=b_0\), then it decreases until it achieves the value
\(\tau_1\) at \(y=2b_0\) etc. Since \(\phi\) and \(\tau\) are \(b\)-periodic, we
have \(b_0=\frac{b}{2q}\) for some \(q\in\mathbb N\) and \(\frac{b}q\) is the
minimal period of \(\phi\). Now separating the variables
in~\eqref{eq:harm-taueq} and applying~\cite[233.00]{BF1971}, we obtain
\[
y=\int_{\tau_1}^{\tau(y)}\frac{\d t}{4\pi\sqrt{P(t)}}=
\frac{1}{2\pi\sqrt{\tau_3-\tau_1}}\sn^{-1}
\left(\sqrt{\frac{\tau-\tau_1}{\tau_2-\tau_1}}\bigmid
\frac{\tau_2-\tau_1}{\tau_3-\tau_1}\right),\qquad
\tau\in [\tau_1,\tau_2],\quad y\in [0,b_0],
\]
and it follows that \(\phi\) is given by~\eqref{eq:intro-phigen}. Substituting
\(y=b_0\), we get~\eqref{eq:intro-taub}.
Using~\eqref{eq:harm-thaleq},~\eqref{eq:harm-cd}, and \(\theta'\ge 0\), we
obtain
\[
\theta(y)=\int_{\tau_1}^{\tau(y)}\frac{\sqrt{\tau_1\tau_2\tau_3}\,\d
t}{2t\sqrt{P(t)}},\quad
\alpha(y)=-\sgn (r+a)\int_{\tau_1}^{\tau(y)}\frac{\sqrt{(1-\tau_1)(1-\tau_2)(\tau_3-1)}
\,\d t}{2(1-t)\sqrt{P(t)}}.
\]
Applying~\cite[233.02]{BF1971}, we obtain that \(\theta\) and \(\alpha\) are
given by~\eqref{eq:intro-theta} and~\eqref{eq:intro-alpha} respectively, and
substituting \(y=b_0\), we get~\eqref{eq:intro-taup},\eqref{eq:intro-taur}. It
follows that \(u=u_{a,b}^{p,q,r}\), where \(p,q,r\)
satisfy~\eqref{eq:intro-pqrineq} by second part of
Proposition~\ref{pr:intro-taus}.

\begin{figure}
\begin{minipage}{0.3\linewidth}
\includegraphics[width=\linewidth]{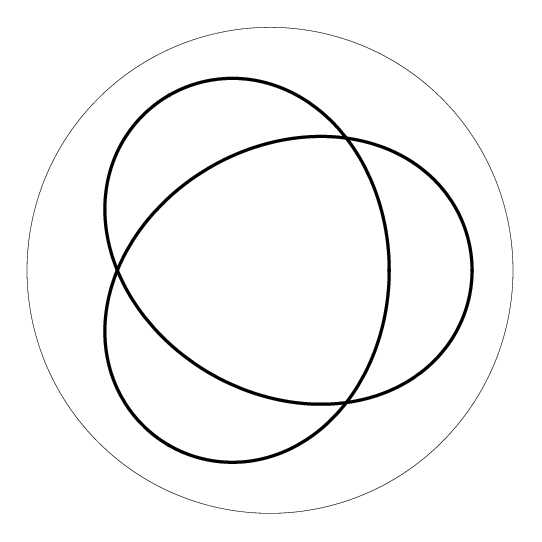}
\end{minipage}
\begin{minipage}{0.3\linewidth}
\includegraphics[width=\linewidth]{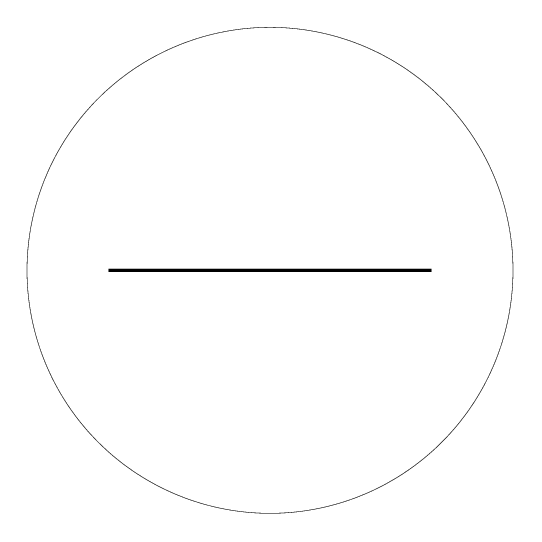}
\end{minipage}
\begin{minipage}{0.3\linewidth}
\includegraphics[width=\linewidth]{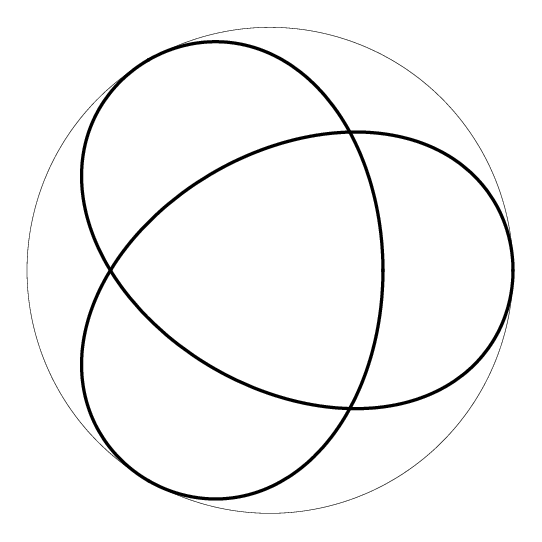}
\end{minipage}
\caption{Examples of the ``top view'' of the orbit space \(\mathbb S_{\ge 0}^2\)
together with the curve \(\gamma_u\). \emph{Left:} the nonlimit case
\(a=0.25,b=2.1,p=2,q=3,r=0;\) \emph{Middle:} the first limit case
\(a=0.25,b=1.25,p=1,q=2,r=0;\) \emph{Right:} the second limit case
\(a=0.5,b=2,p=2,q=3,r=1.\)}
\end{figure}

Now suppose that \(N\in\gamma_u\) but still \(\gamma_u\cap\partial\mathbb S_{\ge
0}^2=\emptyset\). It is easy to see that in this case we have \(c=0\) in the
system~(\ref{eq:harm-phieq},\ref{eq:harm-thaleq}), and for any branch of
solutions at \(N\) we have \(\theta(y)\equiv\theta_0\), where
\(\theta_0\in\mathbb R\) completely characterizes the branch. Moreover, for the
map \(\tilde u\) to be smooth, the branches corresponding to
\(\theta_0,\theta_1\in\mathbb R\) can be glued if and only if
\(\theta_0-\theta_1\in\pi+2\pi\mathbb Z\). Doing this gluing each time
\(\gamma_u\) passes through \(N\), we can express the result as a single
solution with \(\theta\equiv\const\) and \(\phi\) taking values in \((0,\pi)\)
instead of \(\bigl(0,\frac{\pi}2\bigr)\). Note that \(0=:\tau_1\) is a root of
the polynomial \(P(t)\) in this case. If \(\tau_2,\tau_3,q,r\) are as in the
previous paragraph, then the same argument shows (again, possibly after
replacing \(u\) with its pre- and post- compositions with appropriate conformal
diffeomorphism and \(\mathbb S^1\)-equivariant isometry respectively) that
\(\phi,\alpha\) are given by~\eqref{eq:intro-philim1} and~\eqref{eq:intro-alpha}
respectively and~\eqref{eq:intro-taub},\eqref{eq:intro-taur} are satisfied.
For~\eqref{eq:intro-taup}, using~\eqref{eq:taus-bf} and~\eqref{eq:taus-nlim} we
have
\begin{multline}\label{eq:harm-limint}
\lim_{\tau\to 0}\int_\tau^{\tau_2}
\frac{\sqrt{\tau\tau_2\tau_3}\,\d t}{t\sqrt{(t-\tau)(\tau_2-t)(\tau_3-t)}}=
\lim_{\tau\to 0}2\sqrt{\frac{\tau_2\tau_3}{\tau(\tau_3-\tau)}}
\Pi\left(-\frac{\tau_2-\tau}\tau\bigmid
\frac{\tau_2-\tau}{\tau_3-\tau}\right)=\\
\lim_{\tau\to 0}\sqrt{\frac{\tau_2\tau_3}{\tau(\tau_3-\tau)}}\cdot
\frac{\pi}{\sqrt{\frac{\tau_2-\tau}\tau}}=\pi,
\end{multline}
which proves~\eqref{eq:intro-taup} with \(p:=\frac{q}2\) (note that we modified
\(\phi\) so that it is now \(\frac{b}q\)-antiperiodic and \(q\) is always even
in this case). It follows that \(u=u_{a,b}^{p,q,r}\), where \(p,q,r\)
satisfy~\eqref{eq:intro-pqrineq} by second part of
Proposition~\ref{pr:intro-taus}.

Case \(\gamma_u\cap\partial\mathbb S_{\ge 0}^2\ne\emptyset,N\notin\gamma_u\) is
similar but \(\alpha\) plays the role of \(\theta\). Let \(M\in\partial\mathbb
S_{\ge 0}^2\) be any point. We have \(d=0\) in the
system~(\ref{eq:harm-phieq},\ref{eq:harm-thaleq}), and for any branch of
solutions at \(M\) we have \(\alpha(y)\equiv\alpha_0\), where
\(\alpha_0\in\mathbb R\) completely characterizes the branch. Moreover, for the
map \(\tilde u\) to be smooth, the branches corresponding to
\(\alpha_0,\alpha_1\in\mathbb R\) can be glued if and only if
\(\alpha_0-\alpha_1\in\pi+2\pi\mathbb Z\). Doing this gluing any time
\(\gamma_u\) intersects \(\partial\mathbb S_{\ge 0}^2\), we can express the
result as a single solution with \(\alpha\equiv\const\) and \(\phi\) taking
values in \(\bigl(-\frac{\pi}2,\frac{\pi}2\bigr)\) instead of
\(\bigl(0,\frac{\pi}2\bigr)\). Note that \(1=:\tau_2\) is a root of the
polynomial \(P(t)\) in this case. If \(\tau_1,\tau_3,p,q\) are as in the
nonlimit case, then the same argument shows that (again, possibly after
replacing \(u\) with its pre- and post- compositions with appropriate conformal
diffeomorphism and \(\mathbb S^1\)-equivariant isometry respectively)
\(\phi,\theta\) are given by~\eqref{eq:intro-philim2} and~\eqref{eq:intro-theta}
respectively and~\eqref{eq:intro-taub},\eqref{eq:intro-taup} are satisfied.
For~\eqref{eq:intro-taur}, we first note that the \((a,b)\)-periodicity of
\(\tilde u\) and \(\alpha\equiv\const\) imply that \(a\in\mathbb Z\) if \(q\) is
even and \(a\notin\mathbb Z,2a\in\mathbb Z\) if \(q\) is odd.
Then~\eqref{eq:intro-taur} with \(r:=\pm\frac{q}2-a\) can be proved in the same
way as~\eqref{eq:harm-limint}. It follows that \(u=u_{a,b}^{p,q,r}\), where
\(p,q,r\) satisfy~\eqref{eq:intro-pqrineq} by second part of
Proposition~\ref{pr:intro-taus}.

Finally, the ``hybrid'' case \(\gamma_u\cap\partial\mathbb S_{\ge
0}^2\ne\emptyset,N\in\gamma_u\) is similar to the cases above. The details are
left to the reader. We only notice that in this case we cannot glue the branches
of solutions to make \(\phi\) a periodic function but we can make it
quasiperiodic, see~\eqref{eq:intro-philim12}.
\end{proof}

\section{The spectral index of
\texorpdfstring{\(u_{a,b}^{p,q,r}\)}{u\_\{a,b\}\^{}\{p,q,r\}}}\label{sec:weyl}

It follows from~\eqref{eq:harm-P} that
\begin{equation}\label{eq:weyl-A}
A=4\pi^2(\tau_1+\tau_2+\tau_3-1).
\end{equation}
Using this and~\eqref{eq:harm-eu}--\eqref{eq:harm-thaleq},\eqref{eq:harm-cd}, we
see that energy density of the map \(u_{a,b}^{p,q,r}\) rewrites as
\[
e(u_{a,b}^{p,q,r})=\frac{1}2((\phi')^2+((\alpha')^2+4\pi^2)\sin^2\phi+
(\theta')^2\cos^2\phi)=
2\pi^2(-\cos 2\phi+\tau_1+\tau_2+\tau_3-1)=:\rho(y).
\]
In this section we obtain the lower bound~\eqref{eq:intro-N2ineq} on the number
\(N(2)\) for the metric \(g_{a,b}^{p,q,r}\) and prove the sufficient
conditions~\eqref{eq:intro-N2eqsc} and~\eqref{eq:intro-N2eqscsc} for this lower
bound to attain.

The Laplace eigenvalue problem for the metric \(g_{a,b}^{p,q,r}\) is
\begin{equation}\label{eq:weyl-lap}
-(\partial_x^2+\partial_y^2)f=\lambda\rho f,
\end{equation}
where it is convenient to assume that \(f\) is a \emph{complex} valued smooth
function on \(\mathbb R^2/\Gamma_{a,b}\). The following proposition is proved by
the standard separation of variables.

\begin{proposition}\label{pr:weyl-sep}
The (complex) \(\lambda\)-eigenspace of the problem~\eqref{eq:weyl-lap} has a
basis consisting of the functions of the form
\begin{equation}\label{eq:weyl-eigenfun}
h(y;l)e^{2\pi ilx},\qquad l=0,1,\ldots,
\end{equation}
where \(h(y;l)\) solves the following Sturm-Liouville quasiperiodic problem
\begin{gather}
-h''(y)+4\pi^2 l^2 h(y)=\lambda\rho(y)h(y),\label{eq:weyl-sleq}\\
h(y+b)\equiv e^{-2\pi ila}h(y).\qed\label{eq:weyl-slbc}
\end{gather}
\end{proposition}

In general, the eigenvalues of the
problem~(\ref{eq:weyl-sleq},\ref{eq:weyl-slbc}) form a sequence
\[
\lambda_0(l)\le\lambda_1(l)\le\ldots\le\lambda_j(l)\le\ldots\nearrow +\infty.
\]
However, the description of the spectrum can be made more precise if \(l\) is
known. In the following proposition, the eigenfunction corresponding to the
eigenvalue \(\lambda_j(l),j\ge 0\) is denoted by \(h_j\).

\begin{proposition}\label{pr:weyl-sturm}
\begin{enumerate}[label=(\roman*)]
\item
Suppose that \(la\in\mathbb Z\), i.e.~\eqref{eq:weyl-slbc} is the
periodic boundary condition. Then
\[
\lambda_0(l)<\lambda_1(l)\le\lambda_2(l)<\lambda_3(l)\le\lambda_4(l)<\ldots
<\lambda_{2i-1}(l)\le\lambda_{2i}(l)<\ldots\nearrow +\infty.
\]
Moreover, all the eigenfunctions \(h_j,j\ge 0\) can be taken to be real-valued,
the eigenfunction \(h_0\) does not vanish on \([0,b)\), and for each
\(i>0\) each of the eigenfunctions \(h_{2i-1}\) and \(h_{2i}\) has exactly
\(2i\) zeros on \([0,b)\).\label{it:weyl-per}

\item
Suppose that \(la\notin\mathbb Z,2la\in\mathbb Z\), i.e.~\eqref{eq:weyl-slbc} is
the antiperiodic boundary condition. Then
\[
\lambda_0(l)\le\lambda_1(l)<\lambda_2(l)\le\lambda_3(l)<\ldots
<\lambda_{2i}(l)\le\lambda_{2i+1}(l)<\ldots\nearrow +\infty.
\]
Moreover, all the eigenfunctions \(h_j,j\ge 0\) can be taken to be real-valued,
and for each \(i\ge 0\) each of the eigenfunctions \(h_{2i}\) and \(h_{2i+1}\)
has exactly \(2i+1\) zeros on \([0,b)\).\label{it:weyl-aper}

\item
Suppose that \(2la\notin\mathbb Z\). Then
\[
\lambda_0(l)<\lambda_1(l)<\ldots<\lambda_j(l)<\ldots\nearrow +\infty.
\]
Moreover, for each \(j\ge 0\) and \(y_0\in\mathbb R\), both the real and the
imaginary parts of the eigenfunction \(h_j\) have either \(j\) or \(j+1\) zeroes
on \((y_0,y_0+b)\).\label{it:weyl-qper}
\end{enumerate}
\end{proposition}

\begin{proof}
Parts~\ref{it:weyl-per} and~\ref{it:weyl-aper} are versions of the
Sturm-Liouville oscillation theorem and are well-known (see
e.g.~\cite[Chapter~8]{CL1955} and~\cite[Theorem~3.1.2]{Eas1973}).
Concerning~\ref{it:weyl-qper}, consider the Dirichlet eigenvalue problem on
\((y_0,y_0+b)\)
\begin{equation}\label{eq:weyl-sldir}
-h''(y)+4\pi^2 l^2 h(y)=\mu\rho(y)h(y),\qquad
h(y_0)=h(y_0+b)=0
\end{equation}
with eigenvalues
\[
\mu_0(l)<\mu_1(l)<\ldots<\mu_j(l)<\ldots\nearrow +\infty.
\]
Then it is proved in~\cite[Sec.~3.1]{Eas1973} that
\[
\lambda_0(l)<\mu_0(l)<\lambda_1(l)<\mu_1(l)<\ldots<\mu_{j-1}(l)<\lambda_j(l)<
\mu_j(l)<\ldots.
\]
(See also~\cite{EKWZ1997} for a more general statement.) Take some \(j\ge 0\).
Since the coefficients of the equation~\eqref{eq:weyl-sleq} are real, both
\(\Re h_j\) and \(\Im h_j\) are real valued functions
solving~\eqref{eq:weyl-sleq} with \(\lambda=\lambda_j(l)\) (of course, \(\Re
h_j\) and \(\Im h_j\) do not satisfy boundary conditions~\eqref{eq:weyl-slbc}
but it does not matter here). By the classical Sturm oscillation theorem, the
eigenfunction of~\eqref{eq:weyl-sldir} with eigenvalue \(\mu_j(l)\) has exactly
\(j\) zeroes on \((y_0,y_0+b)\). The claim now follows from the Sturm separation
theorem.
\end{proof}

\begin{proposition}\label{pr:weyl-leq0}
We have \(\lambda_{2p-1}(0)=\lambda_{2p}(0)=2\) if \(\frac{p}q\ne\frac{1}2\) and
\(\lambda_{2p-1}(0)<\lambda_{2p}(0)=2\) otherwise.
\end{proposition}

\begin{proof}
We start with the nonlimit case \(\frac{p}q\ne\frac{1}2\). Recall that the
components of the map \(u_{a,b}^{p,q,r}\) are eigenfunctions of the Laplacian
with eigenvalue~2. Then the functions \(\cos\phi\cos\theta\) and
\(\cos\phi\sin\theta\) solve the problem~(\ref{eq:weyl-sleq},\ref{eq:weyl-slbc})
with \(l=0\) and \(\lambda=2\). It is easy to see that both functions have
exactly \(2p\) zeroes on \([0,b)\). By
Proposition~\ref{pr:weyl-sturm}\ref{it:weyl-per} we have
\(\lambda_{2p-1}(0)=\lambda_{2p}(0)=2\).

The limit case \(\frac{p}q=\frac{1}2\) is more subtle. The problem is that
instead of two eigenfunctions \(\cos\phi\cos\theta\) and \(\cos\phi\sin\theta\)
we now have only one eigenfunction \(\cos\phi\). Since this function has exactly
\(2p\) zeroes on \([0,b)\), by Proposition~\ref{pr:weyl-sturm}\ref{it:weyl-per}
we have either \(\lambda_{2p-1}(0)=2\) or
\(\lambda_{2p-1}(0)<\lambda_{2p}(0)=2\). To show that in fact the second
possibility takes place, we use the same idea as in~\cite[sec.~3.3]{Kar2014}
combined with the trick used in~\cite[proof of Proposition~3.12]{Mor2024}.
Consider the following antiperiodic Sturm-Liouville problem
\begin{equation}\label{eq:weyl-slaux1}
-h''(y)=\lambda\rho(y)h(y),\qquad h(y+2b_0)\equiv -h(y),
\end{equation}
where \(b_0:=\frac{b}{2q}\). Let \(\lambda_0'\le\lambda_1'\) be the first two
eigenvalues of the problem~\eqref{eq:weyl-slaux1}. The function
\(\eta(y):=\cos\phi(y)\) solves this problem with \(\lambda=2\) and has exactly
one zero \(y=0\) on \([0,2b_0)\). It follows from the antiperiodic version of
Sturm oscillation theorem that either \(\lambda_0'=2\) or \(\lambda_1'=2\).
Notice that \(\eta(y)\) increases on \([0,b_0]\) while \(\rho(y)\) decreases on
the same interval. Also, both functions are invariant under the transformation
\(y\mapsto 2b_0-y\). Using \(\eta(y+b_0)\) as a test function for
\(\lambda_0'\), we find
\[
\lambda_0'\le
\frac{\int_0^{b_0}\eta'(y+b_0)^2\d y}{\int_0^{b_0}\rho(y)\eta(y+b_0)^2\d y}=
\frac{\int_0^{b_0}\eta'(y)^2\d y}{\int_0^{b_0}\rho(y)\eta(y+b_0)^2\d y}<
\frac{\int_0^{b_0}\eta'(y)^2\d y}{\int_0^{b_0}\rho(y)\eta(y)^2\d y}=2,
\]
where the second inequality follows from the continuous analogue of the
rearrangement inequality, see~\cite[Claim~3.13]{Mor2024}. Hence,
\(\lambda_0'<2\), and the same argument as in~\cite[Proposition~11]{Kar2014}
and~\cite[Proposition~3.15]{Mor2024} shows that
\(\lambda_0'=\lambda_{2p-1}(0),\lambda_1'=\lambda_{2p}(0)\), which concludes the
proof of \(\lambda_{2p-1}(0)<\lambda_{2p}(0)=2\).
\end{proof}

\begin{proposition}\label{pr:weyl-leq1}
Let \(j:=\#\{j'\ge 0\mid\lambda_{j'}(1)<2\}\). Then \(j=\lceil
2|r+a|-1\rceil+\delta_{r+a,0}\).
\end{proposition}

\begin{proof}
Again, we start with the nonlimit case \(\bigl|\frac{r+a}q\bigr|\ne\frac{1}2\).
Similarly to the proof of Proposition~\ref{pr:weyl-leq0}, the function
\(\sin\phi\,e^{i\alpha}\) solves the
problem~(\ref{eq:weyl-sleq},\ref{eq:weyl-slbc}) with \(l=1\) and \(\lambda=2\).
Note that \(\sin\phi\) does not vanish on \([0,b)\) and \(\alpha\) is a
monotonous function satisfying
\[
\alpha(0)=0,\qquad\alpha(b)=-2\pi(r+a).
\]
In particular, the number of zeroes of the function
\[
h_{re}(y):=\sin\phi\cos\alpha=\Re(\sin\phi\,e^{i\alpha})\qquad
\text{(respectively,\quad
\(h_{im}(y):=\sin\phi\sin\alpha=\Im(\sin\phi\,e^{i\alpha})\))}
\]
on \([0,b)\) coincides with the number of zeroes of the function \(\cos z\)
(respectively, \(\sin z\)) on \([0,2\pi|r+a|)\). Consider the following cases.

\begin{itemize}
\item\emph{Case 1: \(a\in\mathbb Z\).} In this case~\eqref{eq:weyl-slbc} is the
periodic boundary condition and the functions \(h_{re},h_{im}\) both solve the
problem~(\ref{eq:weyl-sleq},\ref{eq:weyl-slbc}). We have the following subcases

\begin{itemize}
\item\emph{Subcase 1.1: \(r+a\ne 0\).} Then both functions \(h_{re},h_{im}\)
have \(2|r+a|\) zeroes on \([0,b)\) and it follows from
Proposition~\ref{pr:weyl-sturm}\ref{it:weyl-per} that \(j=2|r+a|-1\).

\item\emph{Subcase 1.2: \(r+a=0\).} Then \(\alpha\equiv 0,h_{im}\equiv 0\), and
\(h_{re}=\sin\phi\) does not vanish on \([0,b)\). It follows from
Proposition~\ref{pr:weyl-sturm}\ref{it:weyl-per} that \(j=0\).
\end{itemize}

\item\emph{Case 2: \(a\notin\mathbb Z,2a\in\mathbb Z\).} In this
case~\eqref{eq:weyl-slbc} is the antiperiodic boundary condition and the
functions \(h_{re},h_{im}\) both solve the
problem~(\ref{eq:weyl-sleq},\ref{eq:weyl-slbc}). Both functions
\(h_{re},h_{im}\) have \(2|r+a|\) zeroes on \([0,b)\) and it follows from
Proposition~\ref{pr:weyl-sturm}\ref{it:weyl-aper} that \(j=2|r+a|-1\).

\item\emph{Case 3: \(2a\notin\mathbb Z\).} Then for each \(y_0\) both functions
\(h_{re},h_{im}\) have either \([2|r+a|]\) or \([2|r+a|]-1\) zeroes on
\((y_0,y_0+b)\) and it follows from
Proposition~\ref{pr:weyl-sturm}\ref{it:weyl-qper} that \(j=[2|r+a|]\).
\end{itemize}

Gathering all together, we get the result for the nonlimit case.

Proceed to the limit case \(\bigl|\frac{r+a}q\bigr|=\frac{1}2\). In this case
\(\alpha\equiv\const\) and \(\sin\phi\) has \(q\) zeroes on \([0,b)\). Since
\(2a\in\mathbb Z\), \eqref{eq:weyl-slbc} is either the periodic or the
antiperiodic boundary condition. By
Proposition~\ref{pr:weyl-sturm}\ref{it:weyl-per},\ref{it:weyl-aper} we have
either \(\lambda_{q-1}(1)<\lambda_q(1)=2\) or
\(\lambda_{q-1}(1)=\lambda_q(1)=2\). Again, consider an auxiliary antiperiodic
Sturm-Liouville problem
\begin{equation}\label{eq:weyl-slaux2}
-h''(y)+4\pi^2 h(y)=\lambda\rho(y)h(y),\qquad h(y+2b_0)\equiv -h(y),
\end{equation}
where \(b_0:=\frac{b}{2q}\). Let \(\lambda_0'\le\lambda_1'\) be the first two
eigenvalues of this problem. The function \(\sin\phi(y)\) solves this problem
with \(\lambda=2\) and has exactly one zero on \([0,2b_0)\). It follows from the
antiperiodic version of Sturm oscillation theorem that either
\(\lambda_0'=2\) or \(\lambda_0'<\lambda_1'=2\). However, now both
\(\sin\phi(y)\) and \(\rho(y)\) decrease on \([0,b_0]\), which means that we
cannot apply the same trick as in the proof of Proposition~\ref{pr:weyl-leq0}.
In fact, we shall show that now the first possibility takes place. To do this,
we observe that the problem~\eqref{eq:weyl-slaux2} rewrites as
\begin{equation}\label{eq:weyl-lame}
-h''(z)+\lambda m\sn^2(z\mid m)h(z)=\left(\frac{\lambda}2\left(
1+\frac{1}{\tau_3-\tau_1}\right)-\frac{1}{\tau_3-\tau_1}\right)h(z),\qquad
h(z+2K(m))\equiv -h(z),
\end{equation}
where \(z:=2\pi\sqrt{\tau_3-\tau_1}y\). This is so-called \emph{Lam\'e
equation}, see, e.g.,~\cite{Vol2025}. Now assume that
\(\lambda_0'<\lambda_1'=2\). Since \(\sin\phi(y)\) is odd, the eigenfunction
corresponding to the eigenvalue \(\lambda_0'\) is even. It follows that the
corresponding solution of~\eqref{eq:weyl-lame} coincides (up to a scalar factor)
with Lam\'e function \(\Ec_\nu^1(z\mid m)\), where \(\nu>0\) satisfies
\(\nu(\nu+1)=\lambda_0'\). Since \(\lambda_0'<2\), we have \(\nu<1\).

We now need the following fact. Although it seems to be well-known, we could not
find a reference.

\begin{claim}
Let \(\nu\in [0,1]\) and \(m\in (0,1)\). Then the function \(\Ec_\nu^1(z\mid
m)\) is increasing on \([0,K(m)]\).
\end{claim}

\begin{proof}
Since it is an independent claim, we use here the notation that is inconsistent
with the rest of the paper but common for the Sturm-Liouville theory. Fix
\(\nu\in [0,1]\) and consider the equation
\begin{equation}\label{eq:weyl-lamegen} -y''(t)+q(t)y(t)=\lambda y(t),\qquad
q(t):=\nu(\nu+1)m\sn^2(t\mid m). \end{equation} The proof uses a standard
technique from Sturm-Liouville theory called the \emph{Pr\"ufer transformation}.
We give only a sketch of the argument here and refer the reader
to~\cite[Chapter~8]{CL1955} and~\cite[Chapter~4, sec.~5]{Zet2005} for the
missing technical details. The Pr\"ufer transformation is given by the
substitution
\[
y(t)=\rho(t)\sin\theta(t),\qquad y'(t)=\rho(t)\cos\theta(t),
\]
that turns the equation~\eqref{eq:weyl-lamegen} into the system
\begin{equation}\label{eq:weyl-theta}
\theta'(t)=\cos^2\theta(t)+(\lambda-q(t))\sin^2\theta(t),\qquad
\rho'(t)=(1+q(t)-\lambda)\rho(t)\sin\theta(t)\cos\theta(t).
\end{equation}
Let \(\theta_\nu(t;\lambda)\) be the solution of the first equation
in~\eqref{eq:weyl-theta} satisfying the initial condition \(\theta(0)=0\) (the
existence of \(\theta_\nu(t;\lambda)\) follows easily from the elementary ODE
theory).

The function \(y_1(t):=\Ec_1^1(t\mid m)=\sn(t\mid m)\) satisfies
\[
-y_1''(t)+2m\sn^2(t\mid m)y_1(t)=(m+1)y_1(t),\qquad
y_1(0)=y_1'(K(m))=0.
\]
Hence the function \(\tilde y_1(t):=y_1\Bigl(\sqrt{\frac{\nu(\nu+1)}2}t\Bigr)\)
satisfies~\eqref{eq:weyl-lamegen} with
\(\lambda=\lambda_1:=\frac{1}2\nu(\nu+1)(m+1)\) and \(\tilde y_1(0)=\tilde
y_1'\Bigl(\sqrt{\frac{2}{\nu(\nu+1)}}K(m)\Bigr)=0\). Since \(\nu\le 1\), we have
\(\sqrt{\frac{2}{\nu(\nu+1)}}\ge 1\) and the function \(\tilde y_1(t)\)
increases on \([0,K(m)]\). This shows that
\(\theta_\nu(t;\lambda_1)\in\left[0,\frac{\pi}2\right]\) for \(t\in [0,K(m)]\).
It is well-known that for each fixed \(t>0\), \(\theta_\nu(t;\lambda)\) is an
increasing function in \(\lambda\) and \(\theta_\nu(t;\lambda)\to +\infty\) as
\(\lambda\to +\infty\). It follows that there exists \(\lambda_2\ge\lambda_1\)
such that \(\theta_\nu(t;\lambda_2)\in\left[0,\frac{\pi}2\right]\) for \(t\in
[0,K(m)]\) and \(\theta_\nu(K(m);\lambda_2)=\frac{\pi}2\). Clearly, this
\(\lambda_2\) is exactly the first eigenvalue of~\eqref{eq:weyl-lamegen} with
boundary conditions \(y(0)=y'(K(m))=0\). Since \(\cos\theta_\nu(t;\lambda_2)\ge
0\) for \(t\in [0,K(m)]\), we see that the corresponding eigenfunction
\(\Ec_\nu^1(t\mid m)\) increases on the same interval.
\end{proof}

It follows that the function \(\Ec_\nu^1(z\mid m)\) increases \([0,K(m)]\),
i.e., the corresponding eigenfunction of~\eqref{eq:weyl-slaux2} increases on
\([0,b_0]\). We can then apply the same trick as in the proof of
Proposition~\ref{pr:weyl-leq0} to show that \(\lambda_0'\) is not the first
eigenvalue of~\eqref{eq:weyl-slaux2}. This contradiction shows that
\(\lambda_0'=2\). Hence, again by the same argument as
in~\cite[Proposition~11]{Kar2014} and~\cite[Proposition~3.15]{Mor2024}, we have
\(\lambda_{q-1}(1)=2\), which immediately implies the result.
\end{proof}

\begin{proposition}\label{pr:weyl-N2eqsc}
Suppose that~\eqref{eq:intro-N2eqsc} is satisfied. Then \(\lambda_0(l)>2\) for
each \(l\ge 2\).
\end{proposition}

\begin{proof}
It follows from the variational characterization of eigenvalues for the
problem~(\ref{eq:weyl-sleq},\ref{eq:weyl-slbc}) that
\[
\lambda_0(l)=\inf\frac{\int_0^b (|h'|^2+4\pi^2 l^2|h|^2)\d y}
{\int_0^b\rho|h|^2\d y},
\]
where the infimum is taken over all \(h\in C^\infty(\mathbb R,\mathbb
C)\setminus\{0\}\) such that \(h(y+b)\equiv e^{-2\pi ila}h(y)\). Since
\[
\rho(y)=2\pi^2(-2\tau+\tau_1+\tau_2+\tau_3)\le 2\pi^2(\tau_2+\tau_3-\tau_1),
\]
for each \(l\ge 2\) we have
\[
\frac{\int_0^b (|h'|^2+4\pi^2 l^2|h|^2)\d y}{\int_0^b\rho|h|^2\d y}
\ge
\frac{16\pi^2\int_0^b |h|^2\d y}{\int_0^b\rho|h|^2\d y}>
\frac{8}{\tau_2+\tau_3-\tau_1}\ge 2.
\]
\end{proof}

\begin{proposition}\label{pr:weyl-N2eqscsc}
If~\eqref{eq:intro-N2eqscsc} is satisfied, then~\eqref{eq:intro-N2eqsc} is
satisfied.
\end{proposition}

\begin{proof}
Rewriting the LHS of~\eqref{eq:intro-N2eqsc} in terms of \(n_0,n_1,m\) we obtain
\[
\tau_2+\tau_3-\tau_1=\frac{n_1}{n_1-n_0}\left(1-n_0-\frac{n_0}m\right).
\]
Consider two cases.

\textit{Case 1: \(\frac{p}q>\frac{1}{\sqrt 3}\).} We have
\[
\frac{n_1}{n_1-n_0}\left(1-n_0-\frac{n_0}m\right)\le
\frac{1}{1-n_0}\left(1-n_0-\frac{n_0}m\right)=
1-\frac{n_0}{(1-n_0)m}.
\]
Hence it suffices to show that \(\Theta(m)>0\), where
\[
\Theta(m):=3+\frac{n_0}{(1-n_0)m}
\]
and \(n_0\) is a function in \(m\) given implicitly by~\eqref{eq:taus-n0n1}. It
follows from~\eqref{eq:taus-limzero} that
\[
\lim_{m\to 0}\Theta(m)=3+\frac{1}{1-\frac{4p^2}{q^2}}>0,\quad\text{since}\quad
\frac{p}q>\frac{1}{\sqrt 3}.
\]
We have to show that \(\Theta(m)\) does not change sign on \((0,1)\). For this
it suffices to show that if \(\Theta(m)=0\) for some \(m\in (0,1)\), then
\[
\Phi(n_0\mid m)=\sqrt{\frac{(1-n_0)(n_0-m)}{n_0}}\Pi(n_0\mid m)<\frac{p\pi}q.
\]
Eliminating \(n_0\) and using again that \(\frac{p}q>\frac{1}{\sqrt 3}\), we
obtain that it suffices to prove the following inequality
\[
\Omega(m):=\sqrt{\frac{4-3m}{3(1-3m)}}\Pi\left(-\frac{3m}{1-3m}\bigmid
m\right)\le\frac{\pi}{\sqrt 3},\qquad m\in (0,1/3).
\]
We have
\[
\Omega'(m)=\frac{2(2-3m)E(m)-(4-3m)(1-m)K(m)}{2m(1-m)\sqrt{3(4-3m)(1-3m)}}.
\]
Since \(\Omega(0)=\frac{\pi}{\sqrt 3}\), it remains to check that
\[
2(2-3m)E(m)-(4-3m)(1-m)K(m)<0,\qquad m\in (0,1/3).
\]
We have
\[
\frac{K(m)}{E(m)}>1+\frac{m}4>\frac{2(2-3m)}{(4-3m)(1-m)},
\]
where the first inequality follows from \(K(0)=E(0)=\pi/2\) and
\[
\frac{\d}{\d m}\left[K(m)-\left(1+\frac{m}4\right)E(m)\right]=
\frac{1}8\left(\frac{(1+3m)E(m)}{1-m}+K(m)\right)>0,
\]
and the second inequality is equivalent to the trivial inequality
\(m^2(3m+5)>0\).

\textit{Case 2: \(\bigl|\frac{r+a}q\bigr|<\frac{\sqrt 3}4\).} It is easy to see
that
\[
\frac{n_1}{n_1-n_0}\left(1-n_0-\frac{n_0}m\right)\le
n_1\left(1+\frac{1}m\right).
\]
Hence it suffices to show that \(\Theta(m)>0\), where
\[
\Theta(m):=4-n_1\left(1+\frac{1}m\right).
\]
and \(n_1\) is a function in \(m\) given implicitly by~\eqref{eq:taus-n0n1}. It
follows from~\eqref{eq:taus-limzero} that
\[
\lim_{m\to 0}\Theta(m)=4-\frac{1}{1-\frac{4(r+a)^2}{q^2}}>0,\quad\text{since}\quad
\left|\frac{r+a}q\right|<\frac{\sqrt 3}4.
\]
We have to show that \(\Theta(m)\) does not change sign on \((0,1)\). For this
it suffices to show that if \(\Theta(m)=0\) for some \(m\in (0,1)\), then
\[
\Phi(n_1\mid m)=\sqrt{\frac{(1-n_1)(n_1-m)}{n_1}}\Pi(n_1\mid m)>
\left|\frac{(r+a)\pi}q\right|.
\]
Eliminating \(n_1\) and using again that \(\bigl|\frac{r+a}q\bigr|<\frac{\sqrt
3}4\), we obtain that it suffices to prove the following inequality
\[
\Omega(m):=\sqrt{\frac{(1-3m)(3-m)}{4(1+m)}}\Pi\left(\frac{4m}{1+m}\bigmid
m\right)\ge\frac{\sqrt 3}4\pi,\qquad m\in (0,1/3).
\]
We have
\[
\Omega'(m)=\frac{(3m^2-2m+3)E(m)-(3-m)(1-m)K(m)}{4m(1-m)\sqrt{(3-m)(1-3m)(1+m)}}.
\]
Since \(\Omega(0)=\frac{\sqrt 3\pi}4\), it remains to check that
\[
(3m^2-2m+3)E(m)-(3-m)(1-m)K(m)>0,\qquad m\in (0,1/3).
\]
We have
\[
\frac{K(m)}{E(m)}<\frac{1}{\sqrt{1-m}}<\frac{3m^2-2m+3}{(3-m)(1-m)},
\]
where the first inequality follows from \(K(0)=E(0)=\pi/2\) and
\[
\frac{\d}{\d m}\left[\sqrt{1-m}K(m)-E(m)\right]=
\frac{(1-\sqrt{1-m})(E(m)-K(m))}{2m\sqrt{1-m}}<0,
\]
and the second inequality is equivalent to the trivial inequality
\[
m(9m^3-11m^2+15m+3)>0,\qquad m\in (0,1/3).
\]
\end{proof}

Now we have
\begin{align*}
N(2)&=\sum_{i=0}^\infty\#\{i>0\mid\lambda_i(0)<2\}+
\sum_{l=1}^\infty\sum_{i=0}^\infty 2\#\{i>0\mid\lambda_i(l)<2\}
&&\text{(by Proposition~\ref{pr:weyl-sep})}\\
&\ge\sum_{i=0}^\infty\#\{i>0\mid\lambda_i(0)<2\}+
\sum_{i=0}^\infty 2\#\{i>0\mid\lambda_i(1)<2\}\\
&=2p-1+\delta_{2p,q}+2(\lceil 2|r+a|-1\rceil+\delta_{r+a,0})
&&\text{(by
Propositions~\ref{pr:weyl-leq0}, \ref{pr:weyl-sturm}\ref{it:weyl-per},
and \ref{pr:weyl-leq1})}.
\end{align*}

This proves~\eqref{eq:intro-N2ineq}. If~\eqref{eq:intro-N2eqscsc} is satisfied,
then the single inequality becomes equality by Propositions~\ref{pr:weyl-N2eqsc}
and~\ref{pr:weyl-N2eqscsc}.\qed
\begin{remark}\label{rem:weyl-N2noteq}
One can show that the inequality~\eqref{eq:intro-N2ineq} can be strict. Indeed,
using \(h(y):=e^{-\frac{4\pi ia}b y}\) as a test function for the
problem~(\ref{eq:weyl-sleq},\ref{eq:weyl-slbc}) with \(l=2\), we obtain
\[
\lambda_0(2)\le
\frac{\int_0^b (|h'|^2+16\pi^2 |h|^2)\d y}{\int_0^b\rho|h|^2\d y}<
\frac{8\bigl(\frac{a^2}{b^2}+1\bigr)}{\tau_1+\tau_3-\tau_2},
\]
where we used that
\[
\rho(y)=2\pi^2(-2\tau+\tau_1+\tau_2+\tau_3)\ge 2\pi^2(\tau_1+\tau_3-\tau_2).
\]
It follows that if
\[
\tau_1+\tau_3-\tau_2=\frac{n_1}{n_1-n_0}\left(1-\frac{n_0}m+n_0\right)>
4\left(\frac{a^2}{b^2}+1\right),
\]
then \(\lambda_0(2)<2\), and the inequality~\eqref{eq:intro-N2ineq} is strict.
This can be achieved by the following step-by-step procedure.
\begin{itemize}
\item
Put
\[
T(m,n_0,n_1):=\frac{n_1}{n_1-n_0}\left(1-\frac{n_0}m+n_0\right).
\]
Choose \(m\in (0,1),\tilde n_0\in (-\infty,0),n_1\in (m,1)\) such that
\(T(m,\tilde n_0,n_1)>4\) (for example, \(m=\frac{1}6,\tilde
n_0=-6,n_1=\frac{9}{10}\)).

\item
Choose \(n_0<0\) so close to \(\tilde n_0\) that \(T_0:=T(m,n_0,n_1)>4\) and
\(\Phi(n_0\mid m)=\frac{p_0\pi}{q_0}\) is a rational multiple of \(\pi\), where
\(\Phi\) is the function defined by~\eqref{eq:taus-Phi}.

\item 
Choose \(k\in\mathbb N\) such that
\[
b:=\frac{kq_0}\pi\sqrt{\left(\frac{1}{n_1}-\frac{1}{n_0}\right)m}K(m)
\]
is so large that \(b>\max\bigl(1,\frac{1}{\sqrt{T_0-4}}\bigr)\). Put
\(p:=kp_0,q:=kq_0\).

\item
Finally, let \(\eta:=\Phi(n_1\mid m)\). Put \(r:=[\eta q],a:=\{\eta q\}\) if
\(\{\eta q\}\in\left[0,\frac{1}2\right]\) and \(r:=[-\eta q],a:=\{-\eta q\}\)
otherwise.
\end{itemize}
With \(a,b,p,q,r\) chosen this way, we have
\[
T(m,n_0,n_1)=T_0>4\left(\frac{1}{4b^2}+1\right)\ge
4\left(\frac{a^2}{b^2}+1\right),
\]
as desired.
\end{remark}

\section{The value of the functional}\label{sec:barlval}

In this section we prove~\eqref{eq:intro-barlval} and~\eqref{eq:intro-barlineq}.
We have
\[
\bar\lambda_{N(2)}(\mathbb T^2,g_{a,b}^{p,q,r})=
2\int_0^b\rho(y)\,\d y=
4\pi^2((\tau_1+\tau_2+\tau_3)b-2\int_0^b\tau\,\d y).
\]
By~\eqref{eq:harm-taueq} and~\cite[233.17 and 331.01]{BF1971} we have
\begin{multline*}
\int_0^b\tau\,\d y=
\frac{q}{2\pi}\int_{\tau_1}^{\tau_2}
\frac{t\,\d t}{\sqrt{(t-\tau_1)(\tau_2-t)(\tau_3-t)}}=
\frac{q\tau_1}{\pi\sqrt{\tau_3-\tau_1}}\left(\left(1-\frac{n_0}m\right)K(m)+
\frac{n_0}m E(m)\right)=\\
b\tau_3-\frac{q}\pi\sqrt{\tau_3-\tau_1}E(m).
\end{multline*}
This implies~\eqref{eq:intro-barlval}. It remains to show that if \(p=q=1\),
then
\[
4\pi^2 b(\tau_1+\tau_2-\tau_3)+8\pi\sqrt{\tau_3-\tau_1}E(m)>
\max\bigl\{\frac{4\pi^2}b,8\pi\bigr\},
\]
which after replacing \(\tau_1,\tau_2,\tau_3\) by \(m,n_0,n_1\) and some
manipulations rewrites as
\[
\left(m-\frac{m}{n_0}-1\right)K(m)^2+2K(m)E(m)>\max\{\pi^2,2\pi b\}.
\]
Introduce the functions
\[
\Xi(m):=\left(m-\frac{m}{n_0}-1\right)K(m)^2+2K(m)E(m),\qquad
\widetilde\Xi(m):=\frac{\Xi(m)}{\sqrt{\Psi(m)}},
\]
where \(\Psi\) is given by~\eqref{eq:taus-Psi} and \(n_0,n_1\) are considered as
functions in \(m\) given implicitly by~\eqref{eq:taus-n0n1}. Then we need to
prove that \(\Xi(m)>\pi^2\) and \(\tilde\Xi(m)>2\). Using~\eqref{eq:taus-nprim},
we find
\[
\Xi'(m)=\frac{n_0E(m)(K(m)-E(m))((m-1)K(m)+E(m))}{m(m-1)
(n_0 E(m)+(m-n_0)K(m))}>0,
\]
where we used~\eqref{eq:taus-EKineq} and~\eqref{eq:taus-EKpos}. Further, it
follows from~\eqref{eq:taus-limzero} that
\[
\lim_{m\to 0}\Xi(m)=\pi^2,
\]
and this proves that \(\Xi(m)>\pi^2\).

The proof of the second inequality is more involved. We have
\[
\widetilde\Xi'(m)=\frac{(mn_0n_1-2mn_0-mn_1+n_0 n_1)E(m)(E(m)-K(m))(E(m)-(1-m)K(m))}
{2(m-1)m\sqrt{\bigl(\frac{1}{n_1}-\frac{1}{n_0}\bigr)m}(n_0 E(m)+(m-n_0)K(m))
(n_1 E(m)+(m-n_1)K(m))}.
\]
Let us show that
\begin{equation}\label{eq:barlval-Xiaux}
mn_0n_1-2mn_0-mn_1+n_0 n_1<0\quad\Leftrightarrow\quad
m-\frac{m}{n_0}+1-2\frac{m}{n_1}>0.
\end{equation}
Since \(n_1>m\), it suffices to show that \(\Theta(m)>0,m\in (0,1)\), where
\[
\Theta(m):=m-\frac{m}{n_0}-1.
\]
It follows from~\eqref{eq:taus-limzero} that
\[
\lim_{m\to 0}\Theta(m)=2>0.
\]
We have to show that \(\Theta(m)\) does not change sign on \((0,1)\). For this
it suffices to show that if \(\Theta(m)=0\) for some \(m\in (0,1)\), then
\[
\Phi(n_0\mid m)\sqrt{\frac{(1-n_0)(n_0-m)}n}\Pi(n_0\mid m)<\pi.
\]
Eliminating \(n_0\), we obtain that it suffices to prove the following
inequality
\[
\Omega(m):=\sqrt{\frac{2-m}{1-m}}\Pi\left(-\frac{m}{1-m}\bigmid m\right)<\pi,
\qquad m\in (0,1).
\]
In sec.~\ref{sec:otsuki} we will see that the function \(\Omega(m)\)
monotonously decreases on \((0,1)\) and
\[
\Omega(0)=\frac{\sqrt 2}2\pi,\qquad
\lim_{m\to 1}\Omega(m)=\frac{\pi}2,
\]
which immediately implies the desired inequality. This finishes the proof
of~\eqref{eq:barlval-Xiaux}. Using this inequality and~\eqref{eq:taus-EKineq}
and~\eqref{eq:taus-EKpos}, we see that \(\widetilde\Xi'(m)<0\). Further, it
follows from~\eqref{eq:taus-limone} and \((m-1)K(m)\to 0\) as \(m\to 1\) that
\[
\lim_{m\to 1}\widetilde\Xi(m)=2,
\]
and this proves that \(\widetilde\Xi(m)>2\). This concludes the proof of
Theorem~\ref{th:intro-main}.\qed

\subsection{Maximal metrics in rectangular conformal classes}

\begin{proof}[Proof of Theorem~\ref{th:intro-rectmax}]
Swapping the sides of the rectangle if necessary, we may assume that \(b>0,b\ne
1\) and \(\rho=\rho(y)\) is independent of \(x\). By rescaling, we may assume
that \(\lambda_1(\mathbb T^2,g)=2\). The Laplace eigenvalue problem for the
metric \(g\) is given by~\eqref{eq:weyl-lap} and Proposition~\ref{pr:weyl-sep}
applies. Since \(a=0\),~\eqref{eq:weyl-slbc} is the periodic boundary condition
for all integer \(l\ge 0\). It follows easily from the variational
characterization of Sturm-Liouville eigenvalues that \(\lambda_0(l)\) is a
monotonously increasing function in \(l\). Since \(\lambda_1(\mathbb T^2,g)=2\)
and \(\lambda_0(0)=0\), we have
\[
\lambda_j(l)>2\qquad\text{whenever}\qquad j=0,\,l>2\quad\text{or}\quad
j=1,\,l>1\quad\text{or}\quad j\ge 2.
\]
In particular, the multiplicity of the eigenvalue \(2\) is at most \(4\) and
there exists a harmonic map \(u\colon (\mathbb T^2,[g_{0,b}])\to\mathbb S^3\)
(possibly, not linearly full). Since the \(\lambda_1\)-eigenspace of
\(\Delta_g\) is generated by functions of the form~\eqref{eq:weyl-eigenfun} with
\(l=0,1\), the map \(u\) can be easily made \(\mathbb S^1\)-invariant w.r.t.
actions~(\ref{eq:intro-actS3},\ref{eq:intro-actT2}). Since we know from
Theorem~\ref{th:intro-flatnotmax} that the flat metric is not maximal in the
conformal class \([g_{0,b}]\) for \(b\ne 1\), it follows from
Theorem~\ref{th:intro-harm} that \(u=u_{0,b}^{p,q,r}\) for some integer
\(p,q>0\) and \(r\) satisfying~\eqref{eq:intro-pqrineq}. Since \(N(2)=1\), the
inequality~\eqref{eq:intro-N2ineq} implies that \(p=q=1,r=0\), and
then~\eqref{eq:intro-pqrineq} implies that \(b>1\). The result follows.
\end{proof}

\section{Relation to Otsuki tori}\label{sec:otsuki}

In this section we prove

\begin{theorem}
For any integers \(\tilde p,\tilde q,\tilde r\) such that \(\tilde p,\tilde q\)
are positive, relatively prime, and satisfy \(\frac{\tilde p}{\tilde
q}\in\bigl(\frac{1}2,\frac{\sqrt 2}2\bigr)\) there exists a unique \(\tilde b\ge
1\) such that for any integer \(k\ge 1\) the map \(u_{-\tilde r,k\tilde
b}^{k\tilde p,k\tilde q,\tilde r}\) is minimal. If \(\frac{\tilde p}{\tilde
q}\ne\frac{1}2\), then the image of this map coincides (up to isometries of
\(\mathbb S^3\)) with the Otsuki torus \(O_{\tilde p/\tilde q}\). These are the
only minimal maps among the maps \(u_{a,b}^{p,q,r}\).
\end{theorem}

\begin{proof}
The map \(u_{a,b}^{p,q,r}\) is minimal if and only if it is conformal, i.e.
\[
|\partial_x u_{a,b}^{p,q,r}|^2=|\partial_y u_{a,b}^{p,q,r}|^2,\qquad
\langle\partial_x u_{a,b}^{p,q,r},\partial_y u_{a,b}^{p,q,r}\rangle=0.
\]
Since
\begin{equation}\label{eq:otsuki-hopf}
|\partial_x u_{a,b}^{p,q,r}|^2=4\pi^2\sin^2\phi,\qquad
|\partial_y u_{a,b}^{p,q,r}|^2=A-4\pi^2\cos^2\phi,\qquad
\langle\partial_x u_{a,b}^{p,q,r},\partial_y
u_{a,b}^{p,q,r}\rangle=2\pi d,
\end{equation}
we obtain that the minimality of \(u_{a,b}^{p,q,r}\) is equivalent to the
equalities \(A=4\pi^2\) and \(d=0\), which in turn are equivalent to
\[
\tau_1+\tau_2=1,\qquad\tau_3=1.
\]
In particular, the second equality and~\eqref{eq:intro-taur} imply that
\(r+a=0\). Rewriting both equalities in terms of \(m,n_0,n_1\), we get
\[
n_0=-\frac{m}{1-m},\qquad n_1=m. 
\]
Eliminating \(n_0\) from the second equation in~\eqref{eq:taus-nm}, we obtain
\[
\Omega(m):=\sqrt{\frac{2-m}{1-m}}\Pi\left(-\frac{m}{1-m}\bigmid m\right)=
\frac{\pi p}q.
\]
We have \(\Omega(0)=\frac{\sqrt 2}2\pi\). Let us compute \(\lim\limits_{m\to
1}\Omega(m)\). We have
\begin{multline*}
\Pi\left(-\frac{m}{1-m}\bigmid m\right)
=\Pi\left(-\frac{1-\epsilon}\epsilon\bigmid 1-\epsilon\right)
=\int_0^{\frac{\pi}2}\frac{\d t}
{\left(1+\frac{1-\epsilon}\epsilon\sin^2 t\right)
\sqrt{1-(1-\epsilon)\sin^2 t}}=\\
\int_0^{\frac{\pi}2}\frac{1}{\cos^3 t}
\frac{\epsilon\,\d t}{(\epsilon+\tan^2 t)\sqrt{1+\epsilon\tan^2 t}}
=\frac{\sqrt\epsilon}2\int_0^\infty\frac{\sqrt{1+\epsilon s}\,\d s}
{(1+s)\sqrt{s(1+\epsilon^2 s)}},
\end{multline*}
where \(\epsilon:=1-m\) and \(s:=\frac{1}\epsilon\tan^2 t\). Hence,
\[
\lim_{m\to 1}\Omega(m)=\lim_{\epsilon\to 0}
\frac{1}2\int_0^\infty\frac{\sqrt{1+\epsilon s}\,\d s}
{(1+s)\sqrt{s(1+\epsilon^2 s)}}=
\frac{1}2\int_0^\infty\frac{\d s}{(1+s)\sqrt s}=
\int_0^\infty\frac{\d x}{1+x^2}=\frac{\pi}2.
\]
Now we have
\[
\Omega'(m)=\frac{2E(m)-(2-m)K(m)}{2m\sqrt{(2-m)(1-m)}}.
\]
Since
\[
\frac{\d}{\d m}[2E(m)-(2-m)K(m)]=\frac{(1-m)K(m)-E(m)}{2(1-m)}<0
\]
by~\eqref{eq:taus-EKineq}, we obtain that \(\Omega(m)\) monotonously decreases
from \(\frac{\sqrt 2}2\pi\) to \(\frac{\pi}2\) implying that
\(\frac{p}q\in\bigl(\frac{1}2,\frac{\sqrt 2}2\bigr)\) and that the value of
\(m\) is uniquely determined by \(\Omega(m)=\frac{p\pi}q\). The value of
\(b\) is then uniquely determined by the first equation in~\eqref{eq:taus-nm}.
We have \(u_{a,b}^{p,q,r}=u_{-\tilde r,k\tilde b}^{k\tilde p,k\tilde q,\tilde
r}\), where \(k:=\mathrm{gcd}(p,q),\tilde p:=\frac{p}k,\tilde
q:=\frac{q}k,\tilde b:=\frac{b}k,\tilde r:=r\).
Since the map \(u_{a,b}^{p,q,r}\) is \(\mathbb S^1\)-equivariant
w.r.t. the actions~(\ref{eq:intro-actS3},\ref{eq:intro-actT2}), its image is an
Otsuki torus by definition. It remains to show that it is exactly the torus
\(O_{\tilde p/\tilde q}\), i.e. it cannot be any other Otsuki torus. Recall that
according to the Hsiang-Lawson construction, Otsuki tori come from closed
geodesics in the orbit space \(\mathbb S_{\ge 0}^2\) equipped with a certain
metric (see~\cite{HL1971,Pen2013b} for the details). On the other hand, if
\(\tilde\phi(y),\tilde\theta(y)\) are the
functions~\eqref{eq:intro-phigen},\eqref{eq:intro-theta} parametrizing the map
\(u_{0,\tilde b}^{\tilde p,\tilde q,0}\), then \(\tilde\phi,\tilde\theta\) also
parametrize the corresponding curve in the orbit space \(\mathbb S_{\ge 0}^2\)
as in sec.~\ref{sec:harm}. If \(\tilde\phi\) is considered as a function of
\(\tilde\theta\), then it is a periodic function of period \(\frac{2\pi\tilde
p}{\tilde q}\). This way the numbers \(\tilde p,\tilde q\) can be inferred from
the curve, and this shows that the image of the map \(u_{0,\tilde b}^{\tilde
p,\tilde q,0}\) is indeed \(O_{\tilde p/\tilde q}\).
\end{proof}

\section{The dependence of \texorpdfstring{\(\bar\lambda_1(\mathbb
T^2,g_{a,b}^{1,1,0})\)}{λ̄₁} on \texorpdfstring{\(a,b\)}{a, b}}
\label{sec:ab}

Although \(\bar\lambda_1(\mathbb T^2,g_{a,b}^{1,1,0})\) is given explicitly
by~\eqref{eq:intro-barlval}, it is not easy to figure out what happens with this
expression when one of the parameters \(a,b\) is fixed and the other one varies.
Indeed, the dependence of \(\tau_1,\tau_2,\tau_3\) on \(a,b\) is very
complicated and implicit. However, it turns out that the derivatives of
\(\bar\lambda_1(\mathbb T^2,g_{a,b}^{1,1,0})\) in \(a\) and \(b\) can be
expressed in terms of \(\tau_1,\tau_2,\tau_3\).

Let \(a,b\in\mathbb R,b>0\) and \(N\in\mathbb N\). The
\textbf{Hopf differential} of a harmonic map \(u\colon (\mathbb
T^2,[g_{a,b}])\to\mathbb S^N\) is given by
\[
\mathcal H:=\langle\partial_z u,\partial_z u\rangle\d z^2,
\]
where \(z:=x+yi\) is a local conformal coordinate on \((\mathbb
T^2,[g_{a,b}])\). The Hopf differential is a globally defined holomorphic
quadratic differential on \((\mathbb T^2,[g_{a,b}])\). In particular, we have
\(\langle\partial_z u,\partial_z u\rangle\equiv H_{re}+H_{im}i\) for some
constants \(H_{re},H_{im}\in\mathbb R\). Theorem~\ref{th:intro-ab} follows
almost immediately from the following proposition, which is of independent
interest.

\begin{proposition}[M.\,Karpukhin, private communication]\label{pr:ab-hopf}
Let \(N\in\mathbb N\) and \(u_{a,b}\colon (\mathbb T^2,[g_{a,b}])\to\mathbb
S^N\) be a family of harmonic maps depending smoothly on \((a,b)\). Then
\[
\d_{a,b}E_{g_{a,b}}(u_{a,b})=2(H_{im}\d a+H_{re}\d b).
\]
\end{proposition}

\begin{proof}
In this proof we identify functions and metrics on \((\mathbb T^2,[g_{a,b}])\)
with \(\Gamma_{a,b}\)-periodic functions and metrics on \(\mathbb R^2\). Since
\((a,b)\) varies, it is convenient to reparametrize \(u_{a,b}\) on the unit
square \(\square:=[0,1]^2\). Put
\[
F_{a,b}:=\begin{pmatrix} 1 & a\\ 0 & b\end{pmatrix},\quad
\hat g_{a,b}:=\frac{1}b F_{a,b}^* g_{a,b}=
\frac{1}b\bigl(\d x^2+2 a\,\d x\d y+(a^2+b^2)\d y^2\bigr),\quad
\hat u_{a,b}=u_{a,b}\circ F_{a,b}.
\]
The inverse metric is given by
\[
\hat g_{a,b}^{-1}=\frac{1}b\bigl((a^2+b^2)\d x^2-2a\,\d x\d y+\d y^2\bigr).
\]
We have
\[
E_{g_{a,b}}(u_{a,b})=E_{\hat g_{a,b}}(\hat u_{a,b})=
\frac{1}{2b}\int_\square\bigl((a^2+b^2)|\partial_x\hat u_{a,b}|^2-
2a\langle\partial_x\hat u_{a,b},\partial_y\hat u_{a,b}\rangle+
|\partial_y\hat u_{a,b}|^2\bigr)\d x\d y.
\]
Hence,
\begin{multline*}
\d_{a,b}E_{g_{a,b}}(u_{a,b})=
\left(\frac{1}b\int_\square (a|\partial_x\hat u_{a,b}|^2-
\langle\partial_x\hat u_{a,b},\partial_y\hat u_{a,b}\rangle)\d x\d y\right)
\d a+\\
\left(\frac{1}{2b^2}\int_\square\bigl((b^2-a^2)|\partial_x\hat u_{a,b}|^2+2a
\langle\partial_x\hat u_{a,b},\partial_y\hat u_{a,b}\rangle-
|\partial_y\hat u_{a,b}|^2\bigr)\d x\d y\right)\d b+
\int_\square g_{a,b}^{-1}(\d\hat u_{a,b},\d(\d_{a,b}\hat u_{a,b}))\d x\d y.
\end{multline*}
For the last term we have
\begin{multline*}
\int_\square g_{a,b}^{-1}(\d\hat u_{a,b},\d(\d_{a,b}\hat u_{a,b}))\d x\d y=
\int_\square\langle\Delta_{g_{a,b}}\hat u_{a,b},\d_{a,b}\hat u_{a,b}\rangle\d x\d y=
\int_\square|\d\hat u_{a,b}|^2\langle\hat u_{a,b},\d_{a,b}\hat u_{a,b}\rangle\d x\d y=\\
\frac{1}2\int_\square|\d\hat u_{a,b}|^2\d_{a,b}\bigl(|\hat u_{a,b}|^2\bigr)\d x\d y=0,
\end{multline*}
since \(|\hat u_{a,b}|^2\equiv 1\).
Hence,
\begin{multline*}
\d_{a,b}E_{g_{a,b}}(u_{a,b})=
\left(\frac{1}b\int_\square (a|\partial_x\hat u_{a,b}|^2-
\langle\partial_x\hat u_{a,b},\partial_y\hat u_{a,b}\rangle)\d x\d y\right)
\d a+\\
\left(\frac{1}{2b^2}\int_\square\bigl((b^2-a^2)|\partial_x\hat u_{a,b}|^2+2a
\langle\partial_x\hat u_{a,b},\partial_y\hat u_{a,b}\rangle-
|\partial_y\hat u_{a,b}|^2\bigr)\d x\d y\right)\d b.
\end{multline*}
On the other hand, since \(u_{a,b}(x,y)=\hat u_{a,b}(x-\frac{a}b y,\frac{y}b)\),
we have
\begin{align*}
4H_{re}&=|\partial_x u_{a,b}|^2-|\partial_y u_{a,b}|^2=
|\partial_x\hat u_{a,b}|^2-\left|-\frac{a}b\partial_x\hat u_{a,b}+\frac{1}b
\partial_y\hat u_{a,b}\right|^2\\&=
\frac{1}{b^2}\bigl((b^2-a^2)|\partial_x\hat u_{a,b}|^2+2a
\langle\partial_x\hat u_{a,b},\partial_y\hat u_{a,b}\rangle-
|\partial_y\hat u_{a,b}|^2\bigr),\\
4H_{im}&=-2\langle\partial_x u_{a,b},\partial_y u_{a,b}\rangle=
-2\langle\partial_x\hat u_{a,b},-\frac{a}b\partial_x\hat u_{a,b}+\frac{1}b
\partial_y\hat u_{a,b}\rangle=
\frac{2}b(a|\partial_x\hat u_{a,b}|^2-\langle\partial_x\hat u_{a,b},
\partial_y\hat u_{a,b}\rangle).
\end{align*}
Comparing the last two equations, we get the result.
\end{proof}

\begin{proof}[Proof of Theorem~\ref{th:intro-ab}]
By~\eqref{eq:otsuki-hopf}, the components of the Hopf differential of the map
\(u_{a,b}^{1,1,0}\) are
\[
H_{re}=\pi^2-\frac{A}4,\qquad H_{im}=-\pi d.
\]
By~\eqref{eq:harm-alper} and~\eqref{eq:harm-thaleq}, for \(p=q=1,r=0\) and
\(a\in\bigl(0,\frac{1}2\bigr)\) we have \(d<0\). It follows that \(H_{im}>0\)
and \(\bar\lambda_1(\mathbb T^2,u_{a,b}^{1,1,0})=2E(u_{a,b}^{1,1,0})\) is
increasing in \(a\) by Proposition~\ref{pr:ab-hopf}. By the same proposition,
that \(\bar\lambda_1(\mathbb T^2,u_{a,b}^{1,1,0})=2E(u_{a,b}^{1,1,0})\) is
decreasing in \(b\) will follow from \(H_{re}<0\) or, equivalently,
\(A>4\pi^2\). Using~\eqref{eq:weyl-A}, it is easy to see that the latter
inequality is equivalent to~\eqref{eq:barlval-Xiaux} proved in
sec.~\ref{sec:barlval}. This observation concludes the proof.
\end{proof}

\section{The third limit case and a nonintegrable Jacobi field}\label{sec:lim3}

In the limit cases of Theorem~\ref{th:intro-main} at least one of the first two
inequalities in~\eqref{eq:intro-pqrineq} turn into equalities. Therefore it is
natural to consider also the \textbf{third limit case} when the third inequality
in~\eqref{eq:intro-pqrineq} becomes equality. More precisely, suppose that
\(p,q,r\) are fixed and \((a',b')\) approaches some point \((a,b)\) such that
\((r+a)^2+b^2=p^2\). Then, since \(\Psi(m)=\frac{\pi^2 b^2}{q^2}\), we have
\(m\to 0\) and by~\eqref{eq:taus-limzero},
\[
-\frac{\tau_1}{\tau_3-\tau_1}=\frac{m}{n_0}\to 1-\frac{4p^2}{q^2},\qquad
\frac{1-\tau_1}{\tau_3-\tau_1}=\frac{m}{n_1}\to 1-\frac{4(r+a)^2}{q^2}.
\]
A short computation now shows that
\[
\tau_1,\tau_2\to\frac{4p^2-q^2}{4b^2},\quad
\tau_3\to\frac{p^2}{b^2},\qquad
\theta(y)\to\frac{2\pi p}b y,\quad
\alpha(y)\to -\frac{2\pi(r+a)}b,
\]
and finally
\begin{equation}\label{eq:lim3-lim}
u_{a',b'}^{p,q,r}\to u_{a,b}^{p,r;\phi_0},\quad\text{where}\quad
\phi_0=\arccos\frac{\sqrt{4p^2-q^2}}{2b}.
\end{equation}
One may wonder what is special about the value of \(\phi_0\) here. It turns out
that for this \(\phi_0\) the map \(u_{a,b}^{p,r;\phi_0}\) has an unexpected
property: \emph{it admits a nonintegrable Jacobi field}. Recall that if
\(u\colon (M,g)\to (N,h)\) is a harmonic map and \((u_t)_{t\in
(-\epsilon,\epsilon)}\) is a smooth variation of \(u=u_0\), then the second
variation of energy at \(t=0\) is given by
\[
\left.\frac{\d^2}{\d t^2}\right|_{t=0}E(u_t)=\langle J_{u,g}V,V\rangle,
\]
where \(V=\left.\frac{\partial u_t}{\partial t}\right|_{t=0}\in \Gamma(u^*TN)\)
and
\[
J_{u,g}V=\Delta^{u^*TN}V-\tr_g R^N(V,\d u)\d u
\]
is the \textbf{Jacobi stability operator}. The fields from \(\ker J_{u,g}\) are
called \textbf{Jacobi fields} along \(u\). Clearly, if the map \(u_t\) is
harmonic for all \(t\in (-\epsilon,\epsilon)\), then \(V\) is a Jacobi field.
Jacobi fields obtained in this way are called \textbf{integrable}.

\begin{proposition}
Let \(a,b\in\mathbb R,b>0\). Let \(p>0,r\) be integers
satisfying~\eqref{eq:intro-preq} and
\(\phi_0:=\arccos\frac{\sqrt{4p^2-q^2}}{2b}\). Then the map
\(u_{a,b}^{p,r;\phi_0}\) admits a nonintegrable Jacobi field.
\end{proposition}

\begin{proof}
Put \(u:=u_{a,b}^{p,r;\phi_0},A:=\cos\phi_0,B:=\sin\phi_0\) for simplicity. Let
\[
\nu_1:=(ie^{\frac{2\pi i}b py},0),\qquad
\nu_2:=(0,ie^{-\frac{2\pi i}b(bx-(r+a)y)}),\qquad
\nu_3:=(-Be^{\frac{2\pi i}b py},Ae^{-\frac{2\pi i}b(bx-(r+a)y)}).
\]
Then \(\nu_1,\nu_2,\nu_3\) form a global basis in \(\Gamma(u^*T\mathbb S^3)\)
orthonormal at each point \((x,y)\in\mathbb (\mathbb T^2,[g_{a,b}])\). A lengthy
but straightforward computation shows that in this basis the operator
\(J_{u,g_{a,b}}\) has the form
\[
J_{u,g_{a,b}}V=\Delta_0 V+4\pi A\begin{pmatrix}
0 & 0 & 0\\
0 & 0 & -1\\
0 & 1 & 0
\end{pmatrix}\partial_x V+\frac{4\pi}b\begin{pmatrix}
0 & 0 & Bp\\
0 & 0 & A(r+a)\\
-Bp & -A(r+a) & 0
\end{pmatrix}\partial_y V,
\]
where \(\Delta_0=-(\partial_x^2+\partial_y^2)\mathbf 1\) is the flat Laplacian
applied to the components of \(V\in\Gamma(u^*T\mathbb S^3)\). Since
\(J_{u,g_{a,b}}\) has constant coefficients, we can split the variables. If
\[
V=\sum_{k,l\in\mathbb Z}V_{k,l}e^{\frac{2\pi i}b (lbx+(k-la)y)},\qquad
V_{k,l}\in\mathbb C^3
\]
is the Fourier expansion of \(V\), then
\[
J_{u,g_{a,b}}^{\mathbb C}V=\frac{4\pi^2}{b^2}
\sum_{k,l\in\mathbb Z}(J_{u,g_{a,b}}^{k,l}V_{k,l})
e^{\frac{2\pi i}b (lbx+(k-la)y)},
\]
where
\[
J_{u,g_{a,b}}^{k,l}=(l^2 b^2+(k-la)^2)\mathbf 1+
2ilb^2 A\begin{pmatrix}
0 & 0 & 0\\
0 & 0 & -1\\
0 & 1 & 0
\end{pmatrix}+2i(k-la)\begin{pmatrix}
0 & 0 & Bp\\
0 & 0 & A(r+a)\\
-Bp & -A(r+a) & 0
\end{pmatrix},\quad k,l\in\mathbb Z.
\]
We have
\[
\det J_{u,g_{a,b}}^{k,l}=(l^2 b^2+(k-la)^2)
((l^2 b^2+(k-la)^2)^2-4A^2(lb^2-(k-la)(r+a))^2-4B^2 p^2(k-la)^2).
\]
Substituting \(k=\pm q,l=0\) and using~\eqref{eq:intro-preq},
\(A^2=\frac{4p^2-q^2}{4b^2},A^2+B^2=1\), we obtain \(\det J_{u,g_{a,b}}^{\pm
q,0}=0\). The corresponding eigenspace of \(J_{u,g_{a,b}}\) is spanned by
\(V_1,V_2\), where
\begin{align*}
V_1&=2Bp\cos\frac{2\pi q}by\,\nu_1+2A(r+a)\cos\frac{2\pi q}by\,\nu_2-
q\sin\frac{2\pi q}by\,\nu_3,\\
V_2&=2Bp\sin\frac{2\pi q}by\,\nu_1+2A(r+a)\sin\frac{2\pi q}by\,\nu_2+
q\cos\frac{2\pi q}by\,\nu_3.
\end{align*}
It remains to show that, for example, \(V_1\) is not integrable. Assume the
converse and let \((u_t)_{t\in (-\epsilon,\epsilon)}\) be the corresponding
harmonic variation of \(u=u_0\). Put \(\lambda:=\frac{2\pi}b
p=\frac{2\pi}b\sqrt{(r+a)^2+b^2}\). Since any harmonic map \((\mathbb
T^2,g_{a,b})\to\mathbb S^3\) is given by eigenfunctions of the Laplacian and
particularly the map \(u\) is given by \(\lambda\)-eigenfunctions, the map
\(u_t\) is also given by \(\lambda\)-eigenfunctions for all small \(t\). In
other words, we may assume that
\[
\tilde u_t(x,y)=
\left(\sum A_{k,l}(t)e^{\frac{2\pi i}b (lbx+(k-la)y)},
\sum B_{k,l}(t)e^{\frac{2\pi i}b (lbx+(k-la)y)}\right),\qquad
A_{k,l},B_{k,l}\colon (-\epsilon,\epsilon)\to\mathbb C,
\]
where, as usual, \(\tilde u_t\colon\mathbb R^2\to\mathbb S^3\) is the lift of
\(u_t\), and the sums are taken over all \(k,l\in\mathbb Z\) satisfying
\begin{equation}\label{eq:lim3-kl}
\frac{4\pi^2}{b^2}(l^2 b^2+(k-la)^2)=\lambda^2.
\end{equation}
We have
\[
2Bp\cos\frac{2\pi q}by=\langle V_1,\nu_1\rangle=
\left\langle\left.\frac{\partial u_t}{\partial t}\right|_{t=0},\nu_1\right\rangle=
\sum\Re(-iA_{k,l}'(0)e^{\frac{2\pi i}b (lbx+(k-la-p)y)}).
\]
By uniqueness of the (real) Fourier expansion, the only nonzero summands in the
RHS are those with \(l=0,k=p\pm q\), and then~\eqref{eq:lim3-kl} implies that
\(q=2p\). A similar argument with \(\nu_2\) in place of \(\nu_1\) shows that
\(q=2|r+a|\). However, we cannot have both equalities at the same time because
of~\eqref{eq:intro-preq}. This contradiction concludes the proof.
\end{proof}

\begin{remark}
The existence of nonintegrable Jacobi fields is quite a rare property. For
example, it is proved in~\cite{Muk1997} that the torus \(u_{0,1}^{1,0;\phi_0}\)
admits a nonintegrable Jacobi field if and only if
\(\phi_0\in\{0,\frac{\pi}6,\frac{\pi}3,\frac{\pi}2\}\).
\end{remark}

\section{Discussion}\label{sec:disc}

\subsection{Why Conjecture~\ref{con:intro-maxmetric} is difficult?}

Once one has a candidate for a maximal metric for \(\bar\lambda_1\) (either
globally or in the conformal class), a standard tool to prove that it is an
actual maximizer is the so-called \emph{Hersch trick}. It was introduced
in~\cite{Her1970} and then successfully applied in many similar situations. The
idea is the following. Consider the family of conformal automorphisms of
\(\mathbb S^3\) given by
\[
G_\gamma(z)=\frac{1-\gamma^2}{|z+\gamma|^2}(z+\gamma)+\gamma,\qquad
\gamma\in\mathbb B^4,\quad z\in\mathbb S^3,
\]
where \(\mathbb B^4=\{\gamma\in\mathbb R^4\colon |\gamma|\le 1\}\). Let
\((a,b)\in\mathcal M\) and \(u\colon (\mathbb T^2,[g_{a,b}])\to\mathbb S^3\) be
a smooth map satisfying
\begin{equation}\label{eq:disc-G}
E(G_\gamma\circ u)\le E(u),\qquad\gamma\in\mathbb B^4.
\end{equation}
Then it follows from the variational characterization of \(\lambda_1\) and a
standard topological argument that
\[
\Lambda_1(M,[g_{a,b}])\le E(u).
\]
In particular, if \(u\) is harmonic, then
\[
\Lambda_1(M,[g_{a,b}])\le\bar\lambda_1(M,\frac{1}2 |\d u|_{g_{a,b}}^2 g_{a,b}).
\]
Therefore, a natural attempt to prove Conjecture~\ref{con:intro-maxmetric} would
be to put \(u:=u_{a,b}^{1,1,0}\) in~\eqref{eq:disc-G}. Unfortunately, it turns
out that in this case~\eqref{eq:disc-G} fails at least for some conformal
classes \((a,b)\in\mathcal M\) (probably, it fails for all of them). Indeed, let
\((a,b)\to (a_0,b_0)\), where \(a_0^2+b_0^2=1\). Then by~\eqref{eq:lim3-lim} we
have
\[
u_{a,b}^{1,1,0}\to u_{a_0,b_0}^{1,0;\phi_0},\quad\text{where}\quad
\phi_0=\arccos\frac{\sqrt 3}{2b_0}.
\]
Let \(\gamma_0:=(1,0,0,0)\). Using~\cite[(1.7)]{MR1986}, we obtain for
\(u:=u_{a_0,b_0}^{1,0;\phi_0}\)
\begin{multline*}
\left.\frac{\d^2}{\d t^2}\right|_{t=0}E(G_{t\gamma_0}\circ u)=
\int_{\mathbb T^2} (3\langle u,\gamma_0\rangle^2-1)
|\d u|_{g_{a,b}}^2\d v_{g_{a,b}}=\\
\frac{4\pi^2}{b_0^2}\int_0^{b_0}
\left(\frac{9}{4b_0^2}\cos^2\frac{2\pi}{b_0} y-1\right)\d y=
\frac{4\pi^2}{b_0^3}\left(\frac{9}8-b_0^2\right)>0.
\end{multline*}
This means that~\eqref{eq:disc-G} fails for
\(u=u_{a,b}^{1,1,0},\gamma=t\gamma_0\), \(t\) sufficiently small, and \((a,b)\)
sufficiently close to \((a_0,b_0)\). Note that in~\cite{ElSIR1996}, the authors
use~\eqref{eq:disc-G} with \(u=u_{a_0,b_0}^{1,0,\psi_0}\), where
\(\psi_0=\arccos\frac{\sqrt 5}{2\sqrt 2}\ne\phi_0\), to prove that the flat
metric is maximal in any rhombic conformal class \((a_0,b_0)\in\mathcal
M,a_0^2+b_0^2=1\).

One may hope to generalize Hersch trick by constructing a four dimensional
family of \emph{energy decreasing variations} of \(u_{a,b}^{1,1,0}\), i.e. such
that the energy of each map of this family does not exceed the energy of
\(u_{a,b}^{1,1,0}\) (cf.~\cite{Kar2024}). However, we think that it would be a
challenging task. Recall that for a harmonic map \(u\colon (\mathbb
T^2,g)\to\mathbb S^3\) its \textbf{energy index} \(\Ind_E u\) (respectively,
\textbf{energy nullity} \(\Nul_E u\)) is the number of negative (respectively,
zero) eigenvalues of the corresponding Jacobi stability operator \(J_{u,g}\).
Clearly, any family of variations as above induces a four dimensional subspace
\(U\subset\Gamma(u^*T\mathbb S^3)\) such that \(J_{u,g}|_U\le 0\). However,
according to our numerical experiments in Wolfram Mathematica, it seems that
\(\Ind_E(u_{a,b}^{1,1,0})=3\) and \(\Nul_E(u_{a,b}^{1,1,0})=7\) (in fact, we
have \(\Ind_E(u_{a,b}^{1,1,0})\le 4N(2)=4\) by~\cite[Proposition~1.6]{Kar2021}
anyway). Six Jacobi fields come from the rotations of \(\mathbb S^3\) (these
fields are not really useful) and the remaining field comes from the conformal
automorphisms of \((\mathbb T^2,[g_{a,b}])\) given by \(y\)-shifts. Therefore,
we have very few fields contributing to the energy index, and it seems that they
do not have any clear geometrical meaning. This makes difficult even the
construction of the space \(U\), let alone a four dimensional family of energy
decreasing variations.

Finally, note that the analogs of~\eqref{eq:disc-G} as well as the absence of
nontrivial Jacobi fields both play an essential role in establishing the
quantitative stability of maximizers in a conformal class, see~\cite{KNPS2025}.
In particular, the above discussion shows that if
Conjecture~\ref{con:intro-maxmetric} is true, then the methods
of~\cite{KNPS2025} cannot be applied to prove the quantitative stability of the
corresponding maximizer (however, it is possible that this stability can be
treated by other methods).

\subsection{Directions for the future research}

It would be interesting to establish an analogue of Theorem~\ref{th:intro-harm}
for Klein bottles. Another possibility is to replace the
action~\eqref{eq:intro-actS3} by
\[
\alpha\cdot (z_1,z_2)=(e^{2\pi ik\alpha}z_1,e^{2\pi il\alpha}z_2),\qquad
(z_1,z_2)\in\mathbb S^3,\qquad
\alpha\in\mathbb R/\mathbb Z\cong\mathbb S^1
\]
for some positive relatively prime integers \(k,l\). The corresponding \(\mathbb
S^1\)-equivariant \emph{minimal} tori and Klein bottles are called \emph{Lawson
\(\tau\)-surfaces}~\cite{Law1970,HL1971,Pen2012}. Note that, unlike Otsuki tori,
the Lawson \(\tau\)-surfaces are given by elementary functions. It is reasonable
to expect that the harmonic analogues of Lawson \(\tau\)-surfaces (if they
exist) are given by elliptic integrals. The case of higher-dimensional spheres
is of interest as well (see~\cite{Pen2015} for the case of \(\mathbb
S^1\)-equivariant minimal tori in \(\mathbb S^5\)). Finally, it would be
interesting to compare Theorem~\ref{th:intro-harm} with the classification of
harmonic \(\mathbb S^1\)-equivariant tori in terms of spectral
data~\cite{CO2020}.

\bibliographystyle{alpha}
\bibliography{mybib.bib}

\end{document}